\documentclass[9pt,leqno,a4paper]{article}
\usepackage{color}
\usepackage{amsfonts}
\usepackage{amsmath, amsfonts, amsthm, amssymb, amscd}
\usepackage[mathcal]{euscript}
\usepackage{mathrsfs}
\usepackage{dsfont}
\usepackage{ulem}
\usepackage{bbm}

\usepackage{slashed}
\usepackage{a4wide}

\usepackage{lineno} 

\newtheorem{theorem}{Theorem}[section]
\newtheorem{proposition}[theorem]{Proposition}
\newtheorem{lemma}[theorem]{Lemma}
\newtheorem{corollary}[theorem]{Corollary}

\newtheorem{remark}[theorem]{Remark}

\numberwithin{equation}{section}

\usepackage{multirow}
\usepackage{tabularx}
\usepackage{lscape}

\newcommand \Dcal{\mathcal{D}}
\newcommand \Kcal {\mathcal K}

\newcommand \Hcal {\mathcal H}

\newcommand \Hinit{\Hcal^{\text{\bf init}}}
\newcommand \Jcal {\mathcal{J}}
\newcommand \Lcal {\mathcal L}
\newcommand \Pcal {\mathcal P}

\newcommand \delb {\bar {\del}}

\newcommand \Hb{\bar H}

\newcommand \ust{u^{\star}}
\newcommand \vst{v^{\star}}

\newcommand \Gammat {\widetilde{\Gamma}}
\newcommand \Deltat {\widetilde{\Delta}}

\newcommand \del \partial
\newcommand \delu {\uline{\del}}

\newcommand \Tu {\uline{T}}
\newcommand \Hu {\uline{H}}

\newcommand \Pu{\uline{P}}
\newcommand \Psiu{\uline{\Psi}}
\newcommand \Phiu{\uline{\Phi}}

\newcommand \Abf {{\bf A}}
\newcommand \Bbf {{\bf B}}
\newcommand \Ebf {{\bf E}}
\newcommand \Fbf {{\bf F}}
\newcommand \init{\text{\bf init}}
\newcommand \source{\text{\bf sour}}

\newcommand \RR{\mathbb{R}}
\newcommand \eps{\varepsilon}

\newcommand {\vep}{\varepsilon}

\newcommand {\dels}{\slashed{\del}}

\newcommand {\ebf}{{\bf e}}

\newcommand{\Vscr}{\mathscr{V}}

\makeatletter
\def\hlinew#1{%
	\noalign{\ifnum0=`}\fi\hrule \@height #1 \futurelet
	\reserved@a\@xhline}
\makeatother

\title{A rigidity property for a type of wave-Klein-Gordon system\footnote{The present work is supported by Shaanxi Fundamental Science Research Project for Mathematics and Physics (Grant No.22JSZ003)}}

\author{Yan-Tao Li \footnote{School of Mathematics and  Statistics, Xi'an Jiaotong University, Xi'an, Shaanxi 710049, P.R. China.\ E-mail:lyt666@stu.xjtu.edu.cn}, 
Yue MA \footnote{School of Mathematics and Statistics, Xi'an Jiaotong University, Xi'an, Shaanxi 710049, P.R. China.\ E-mail: yuemath@xjtu.edu.cn}
}

\begin{document}
	\maketitle
	
\begin{abstract}
In this paper we investigate the rigidity property of a wave component coupled in a wave-Klein-Gordon system. We prove that when the radiation field of the wave component vanishes at the null infinity, the initial data of this component also vanish, therefor there is no wave in the whole spacetime. 
\end{abstract}	
	\section{Introduction}
In the present work we establish a rigidity result of the following Cauchy problem\footnote{In this papaer $\Box = \del_t\del_t - \sum_{a=1}^3\del_a\del_a$} of a wave-Klein-Gordon system in $\RR^{1+3}$: 
	\begin{equation}\label{eq4-13-05-2021}
		\aligned
		&\Box u = B^{\alpha\beta}\del_{\alpha}u\del_{\beta}v,
		\\
		&\Box v + c^2v = P^{\alpha\beta}u\del_{\alpha}\del_{\beta}v
		\endaligned
	\end{equation}
with the small localized regular initial data imposed on $\{t=2\}$:
\begin{equation}\label{eq19-19-08-2022-M}
u(2,x) =u_0(x),\quad \del_t u(2,x) = u_1(x),\quad v(2,x) = v_0(x),\quad \del_t v(2,x) = v_1(x).
\end{equation}
	Here $B^{\alpha\beta}, P^{\alpha\beta}$ are constant-coefficient quadratic forms, $c$ is a constant, $c>0$, and $u_{\ell}$ and $v_{\ell}$ ($\ell = 0,1$) are supported in the unit disc $\{|x|<1\}$ with sufficient regularity. The term ``rigidity'' means that, when we cannot detect the wave at the null infinity via a radiation field, there is no wave in the whole of spacetime.
	
	We firstly revisit the global well-posedness result and show the existence of the Friedlander radiation field of the wave component:
	\begin{equation}\label{eq2-29-dec-2023}
		\mathcal{R}_{u_t}(\mu,\omega) :=  \lim_{r \rightarrow \infty} (r\del_t u)(r+\mu ,r\omega)
	\end{equation}
	in the regime of the small amplitude global regular solution, see also \cite{Baskin-Wang-2014, Jia-Li-2023}. More precisely, we will firstly prove the following result:
	\begin{theorem}[Global well-posedness and existence of Friedlander radiation field]\label{thm1-14-05-2021}
	Consider the Cauchy problem associate to \eqref{eq4-13-05-2021} with the initial data \eqref{eq19-19-08-2022-M}. There exists a positive constant $\vep_0>0$ and an integer $N\geq 9$ (both determined by the system) such that when $\vep < \vep_0$ and
	\begin{equation}\label{eq1-28-dec-2023}
	\aligned
	\|\del_xu_0\|_{H^N(\RR^3)} + \|u_1\|_{H^N(\RR^3)} \leq \eps,
	\\
	\|v_0\|_{H^{N+1}(\RR^3)} + \|v_1\|_{H^N(\RR^3)}\leq \eps,
	\endaligned
	\end{equation}
	the associate local solution is in fact well-defined in $\{t\geq 2\}$.  In this case,
	the Friedlander radiation field $\mathcal{R}_{u_t}$ associate to $u$ is well defined. 
\end{theorem} 
\begin{remark}
The regularity condition $N\geq 9$ is far from optimal. In fact a more careful application of the method presented later can make an improvement to $N\geq 7$. But in order to concentrate on the rigidity result, we will not make efforts in this direction. 
\end{remark}
Then we prove the rigidity, that is, given the initial data sufficiently small and $\mathcal{R}_{u_t}(\mu,\omega) \equiv 0$, the wave component vanishes identically. More precisely,

\begin{theorem}[Rigidity for small amplitude regular solution]\label{thm1-28-dec-2023}
	Let $(u,v)$ be the global solution to the Cauchy problem \eqref{eq4-13-05-2021}-\eqref{eq19-19-08-2022-M}.
	There exists a positive constant $\vep_1>0$, determined by the system, such that if  $\vep<\vep_1$ and  $\mathcal{R}_{u_t}\equiv 0$, then $u_0 = u_1 = 0$.	
\end{theorem}

\begin{remark}
One may observe that when $u\equiv 0$, the system reduces to the free-linear Klein-Gordon equation. Thus the vanishing radiation field can only eliminate the wave component. 
\end{remark}

The rigidity property is important for many PDE systems from physics because people believe that when there is no radiation towards the infinity (in certain sens), there should be no wave in the whole spacetime. In this direction there are plenty of researches. On electromagnetic (Maxwell equations) fields this dates back to at least \cite{Papapetrou-1957} and for gravitational field (Einstein equations) \cite{Papapetrou-1965}. In the mathematical literature, this research is sometimes called the problem of unique continuation form infinity. This was initiated in \cite{Lax-Phillips-1967} and followed in \cite{Friedlander-1973,Friedlander-1980}. See also more recent works on linear/nonlinear wave equations \cite{Alexakis-Shao-2015,Alexakis-Schlue-Shao-2016,Lindblad-Schlue-2023}, on Dirac equations \cite{Jia-Li-2023}, and in plasma physics \cite{Li-Yu-2021}, and so on. 

In the present paper, we investigate the rigidity property of the wave component coupled in a wave-Klein-Gordon system. We have chosen \eqref{eq4-13-05-2021} as a simple model for this stating step. In mixed quadratic terms coupled in the wave equation was extensively studied in many context such as \cite{Dong-2019, Dong-2020, Ifrim-Sting-2019}. The nonlinear coupling in the Klein-Gordon equation is the typical quasi-linear term included in the wave-Klein-Gordon model introduced and studied in \cite{PLF-YM-CMP,IP3, Chen-Kindblad-2023}. 

The mechanism relies on an enforced energy estimate on wave component, which is carried out by applying the conformal-killing (with respect to the Minkowski metric) multiplier $Su = (t\del_t + r\del_r)u$. This energy (denoted by $\Ebf_1$ in the coming sections) is stronger than the standard energy and provides strong decay rate on all components of $\nabla_{t,x} u$ except one bad direction (this can be $\underline{L} = \del_t-\del_r$ or $\del_r = (x^a/r)\del_a$. In the present paper we take $\del_t$ as the bad direction). Then combined with an ODE-based argument along a group of time-like hyperbolas together, one can establish an decay rate on $\nabla_{t,x} u$ by integrating form the null infinity. This decay rate, as we call it ``excessive decay'', is even faster than the generic linear wave decay. This leads to the fact that the standard energy (obtained by the multiplier $\del_t u$) of the wave component tends to zero at a polynomial rate. Then an argument based on the energy identity and Gronwall's inequality shows that the initial standard energy vanishes, and we conclude by the desired result.

	\section{Some basic facts of hyperboloidal framework}\label{sec1-13-05-2021}
	\subsection{Frames and vector fields}
	Let $(t,x)\in \RR^{1+3}$ with $x\in \RR^3$. Denote by $r = \sqrt{|x^1|^2+|x^2|^2 +|x^3|^2}$. We work in the light-cone $\Kcal := \{r < t-1\}\subset \RR^{1+3}$. 
	Furthermore, we introduce
	$$
	\Hcal_s := \{(t,x)\in \RR^{1+3}| t = \sqrt{s^2+r^2}\}, \quad \Hcal_{[s_0,s_1]} := \Kcal\cap \Big\{\sqrt{s_0^2+|x|^2}\leq t\leq \sqrt{s_1^2+|x|^2}\Big\},
	$$
	$$
	\del\Kcal = \{t = |x|+1\},\quad \del\Kcal_{[s_0,s_1]} = \del\Kcal\cap \Big\{\sqrt{s_0^2+|x|^2}\leq t\leq \sqrt{s_1^2+|x|^2}\Big\}.
	$$
	and
	$$
	\Hinit_{s_0} := \big(\Hcal_{s_0}\cap \Kcal\big)\cup \del\Kcal_{[s_0,s_1]}.
	$$

	We recall the following nations introduced in \cite{PLF-YM-book1}:
	$$
	\delu_0:=\del_t,\quad \delu_a := \delb_a = (x^a/t)\del_t + \del_a.
	$$
	
	The transition matrices between this frame and the natural frame $\{\del_{\alpha}\}$ are:
	\begin{equation}\label{eq semi-frame}
		\Phiu_{\alpha}^{\beta} := \left(
		\begin{array}{cccc}
			1 &0 &0 &0
			\\
			x^1/t &1 &0 &0
			\\
			x^2/t &0 &1 &0
			\\
			x^3/t &0 &0 &1
		\end{array}
		\right),
		\quad
		\Psiu_{\alpha}^{\beta} := \left(
		\begin{array}{cccc}
			1 &0 &0 &0
			\\
			-x^1/t &1 &0 &0
			\\
			-x^2/t &0 &1 &0
			\\
			-x^3/t &0 &0 &1
		\end{array}
		\right)
	\end{equation}
	with
	$$
	\delu_{\alpha} = \Phiu_{\alpha}^{\beta}\del_{\beta},\quad \del_{\alpha} = \Psiu_{\alpha}^{\beta}\delu_{\beta}.
	$$
	
	The vector field (derivatives) $\delu_a$ are tangent to the hyperboloid $\Hcal_s$. We call them hyperbolic derivatives.
	
	Let $T = T^{\alpha\beta}\del_{\alpha}\otimes\del_{\beta}$ be a two tensor defined in $\Kcal$ or its subset. Then $T$ can be written with $\{\delu_{\alpha}\}$:
	$$
	T = \Tu^{\alpha\beta} \delu_{\alpha}\otimes\delu_{\beta}
	\quad\text{with}\quad 
	\Tu^{\alpha\beta} = T^{\alpha'\beta'}\Psiu_{\alpha'}^{\alpha}\Psiu_{\beta'}^{\beta}.
	$$

	\subsection{High-order derivatives}
	In the region $\Kcal$, we introduce the following Lorentzian boosts:
	$$
	L_a = x^a\del_t + t\del_a,\quad a = 1, 2, 3
	$$
	and the following notation of high-order derivatives: let $I,J$ be multi-indices taking values in $\{0,1,2,3\}$ and $\{1,2,3\}$ respectively,
	$$
	I = (i_1,i_2,\cdots, i_m),\quad J = (j_1,j_2,\cdots, j_n).
	$$
	We define
	$$
	\del^IL^J = \del_{i_1}\del_{i_2}\cdots \del_{i_m}L_{j_1}L_{j_2}\cdots L_{j_n}
	$$
	to be an $(m+n)-$order derivative. 
	
	Let $\mathscr{Z}$ be a family of vector fields. $\mathscr{Z} = \{Z_i|i=0,1,\cdots, 6\}$ with
	$$
	Z_0 = \del_t,\quad  Z_1=\del_1,\quad Z_2 = \del_2,\quad Z_3 = \del_3,\quad Z_4= L_1,\quad Z_5=L_2,\quad Z_6 = L_3,
	$$
	A high-order derivative of order $N$ on $\mathscr{Z}$ with milti-index $I = (i_1,i_2,\cdots, i_N)$, $i_j\in\{0,1,\cdots,6\}$ is defined as
	$$
	Z^I := Z_{i_1}Z_{i_2}\cdots Z_{i_N}.
	$$
	A high-order derivative $Z^I$ is said to be of type $(a,b)$, if it contains at most $a$ partial derivatives and $b$ Lorentzian boosts. We then introduce the following notation:
	$$
	\mathcal{I}_{p,k} = \{I| I \text{ is of type }(a,b)\text{ with }a+b \leq p, b\leq k \}
	$$
	so $Z^I$ with $I\in \mathcal{I}_{p,k}$ stands for a high-order derivatives composed by boosts and partial derivative. Its order is smaller or equal to $p$ and it contains at most $k$ boosts.  We define
	\begin{equation}\label{eq1 notation}
		\aligned
		|u|_{p,k} &:= \max_{K\in \mathcal{I}_{p,k}}|Z^K u|,\quad &&|u|_p := \max_{0\leq k\leq p}|u|_{p,k},
		\\
		|\del u|_{p,k} &:= \max_{\alpha}|\del_{\alpha} u|_{p,k}, &&|\del u|_p := \max_{0\leq k\leq p}|\del u|_{p,k},
		\\
		|\del^m u|_{p,k} &:= \max_{|I|=m}|\del^I u|_{p,k}, &&|\del^m u|_p := \max_{0\leq k\leq p}|\del^I u|_{p,k},
		\\
		|\dels u|_{p,k} &:= \max\{|\delu_a u|_{p,k}\}, &&|\dels u|_p := \max_{0\leq k\leq p}|\dels u|_{p,k},
		\\
		|\del\dels  u|_{p,k} &:=\max_{a,\alpha} \{|\delu_a\del_{\alpha} u|_{p,k},|\del_{\alpha}\delu_a u|_{p,k}\},
		&&| \del\dels u|_p :=\max_{0\leq k\leq p}| \del\dels u|_{p,k},
		\\
		|\dels\dels  u|_{p,k} &:=\max_{a,b} \{|\delu_a\delu_b u|_{p,k}\},
		&&| \dels\dels u|_p :=\max_{0\leq k\leq p}| \dels\dels u|_{p,k}.
		\endaligned
	\end{equation}
	In the above expressions the $\max$ is taken for $\alpha\in\{0,1,2,3\}$ and $a,b\in\{1,2,3\}$.
	These quantities will be applied in order to control varies of high-order derivatives in the following discussion. Our first task is to bound them by energy densities. These results will be stated after the introduction of the standard and scaling energy inequalities in the following two sections.
		
	\section{The energy estimate with scaling multiplier on hyperboloids}\label{sec2-13-05-2021}
	\subsection{Differential identities}
We recall the following energy estimate (see for example the applications on Maxwell equations in \cite{Psarelli-2005,Fang-W-Y-2021}).Let $K_1: = t\del_t + x^a\del_a$. We firstly write the following differential identities in $\RR^{1+n}$:
	$$
	\aligned
	K_1 u \Box u =\, &\frac{1}{2}\del_t\Big(t\sum_{\alpha=0}^n |\del_{\alpha}u|^2 + 2x^b\del_bu \del_t u\Big)
	\\
	&- \frac{1}{2}\del_a\Big(2t\del_tu\del_au + 2x^b\del_bu\del_au + x^a|\del_t u|^2 - x^a\sum_b|\del_bu|^2\Big)
	\\
	&+ \frac{n-1}{2}\Big(|\del_t u|^2 - \sum_b|\del_b u|^2\Big),
	\\
	u\Box u = &\,\del_t(u\del_tu) - \del_a(u\del_a u) - \Big(|\del_t u|^2 - \sum_b|\del_b u|^2\Big).
	\endaligned
	$$
	This leads to the divergence form:
	\begin{equation}\label{eq7-28-04-2021}
		\aligned
		&\big(K_1u + (n-1)u/2\big)\Box u
		\\ = &\frac{1}{2}\del_t\Big(t\sum_{\alpha=0}^n |\del_{\alpha}u|^2 + 2x^b\del_bu \del_t u + (n-1)u\del_t u\Big)
		\\
		&-\frac{1}{2}\del_a\Big(2t\del_tu\del_au + 2x^b\del_bu\del_au + x^a|\del_t u|^2 - x^a\sum_b|\del_bu|^2 + (n-1)u\del_a u\Big).
		\endaligned
	\end{equation}
	\subsection{Energy density}
	For simplicity of expression, we write \eqref{eq7-28-04-2021} into the following form
	$$
	\big(K_1u + (n-1)u/2\big)\Box u = \text{Div} \Vscr_1[u]
	$$
	with
	$$
	\aligned
	2\Vscr_1^0[u] := & t\sum_{\alpha=0}^d |\del_{\alpha}u|^2 + 2x^b\del_bu \del_t u + (n-1)u\del_t u,
	\\
	-2\Vscr_1^a[u] := & 2t\del_tu\del_au + 2x^b\del_bu\del_au + x^a|\del_t u|^2 - x^a\sum_b|\del_bu|^2 + (n-1)u\del_a u.
	\endaligned
	$$
	In order to do energy estimate on hyperboloids, we need to regard firstly the energy density $\Vscr_1[v]\cdot (1,-x^a/t)$, which is written as 
	$$
	\aligned
	2\ebf_1[u]:=& t\big(1+(r/t)^2\big)|\del_t u|^2 + t\big(1-(r/t)^2\big)\sum_a|\del_a u|^2 + 2t^{-1}|x^a\del_a u|^2 + 4x^a\del_tu\del_au
	\\
	& + \frac{n-1}{t}uK_1u.
	\endaligned
	$$
	It can be written as
	\begin{equation}\label{eq1-27-04-2021}
		\ebf_1[u]= \frac{1}{2t}|K_1u|^2 + \frac{t}{2}|(x^a/r)\delu_a u|^2 + \frac{s^2}{2t}\sum_{a<b}|r^{-1}\Omega_{ab} u|^2 + \frac{n-1}{2t}uK_1u
	\end{equation}
	where $\Omega_{ab} := x^a\del_b - x^b\del_a$. When $n=1$ the last two terms disappear and when $n=2$ we denote by $\Omega = \Omega_{12}$. Then in order to guarantee the positivity, we establish the following Hardy type inequality:
	\begin{lemma}\label{lem1-30-04-2021}
		Let $u$ be a compactly supported, sufficiently regular function defined on $\RR^n$. Then for $\alpha<n$,
		\begin{equation}\label{eq1-30-04-2021}
			\|r^{-\alpha/2}u\|_{L^2(\RR^n)}\leq \frac{2}{n-\alpha}\|r^{1-\alpha/2}\del_ru\|_{L^2(\RR^n)}
		\end{equation}
		where $r = \sqrt{\sum_{a=1}^n|x^a|^2}$ and $\del_r : = (x^a/r)\del_a$. 
	\end{lemma}
	\begin{remark}
		If we take $\alpha = 2$, \eqref{eq1-30-04-2021} leads to
		$$
		\|r^{-1}u\|_{L^2(\RR^n)}\leq \frac{2}{n-2}\|\del_r u\|_{L^2(\RR^n)}	
		$$
		which is known as the classical Hardy's inequality for $n\geq 3$. In the present work we will apply the case $\alpha=1$ and $n\geq 2$.
	\end{remark}
	\begin{proof}
		Remark the following identity
		$$
		u^2 r^{-\alpha} = \frac{u^2}{n-\alpha}\del_a(x^a/r^{\alpha}) = \frac{1}{n-\alpha}\del_a\big(u^2x^ar^{-\alpha}\big) - \frac{2u}{n-\alpha}(x^ar^{-\alpha})\del_au
		$$
		Integrate the above identity in the domain $\{r \geq \eps\}$ and apply Stokes' formula ($u$ is compactly supported), we obtain
		$$
		\aligned
		\int_{r\geq \eps}u^2 r^{-\alpha}dx =& \frac{1}{n-\alpha}\int_{r\geq \eps}\del_a\big(u^2x^ar^{-\alpha}\big)dx
		- \frac{2}{n-\alpha}\int_{r\geq\eps}u(x^ar^{-\alpha})\del_au\, dx
		\\
		=& \frac{1}{n-\alpha} \int_{r=\eps} u^2(x^ar^{-\alpha})\cdot (-x^a/r)d\sigma - \frac{2}{n-\alpha}\int_{r\geq\eps}u(x^ar^{-\alpha})\del_au\, dx
		\endaligned
		$$
		where $d\sigma = \eps^{n-1}d\omega_{n-1}$ is the volume form of $\{r = \eps\}$ while $d\omega_{n-1}$ is that of the unit $n-1$ sphere. Then we obtain
		$$
		\int_{r\geq \eps}u^2 r^{-\alpha}dx = \int_{r=\eps} u^2 r^{n-\alpha} d\omega_{n-1} - \frac{2}{n-\alpha}\int_{r\geq\eps}u(x^ar^{-\alpha})\del_au\, dx.
		$$
		Let $\vep\rightarrow 0$, we obtain
		$$
		\|r^{-\alpha/2}u\|_{L^2(\RR^n)}^2 
		\leq 
		-\frac{2}{n-\alpha}\int_{\RR^n}r^{-\alpha/2}u\, x^ar^{-\alpha/2}\del_au dx \leq \frac{2}{n-\alpha}\|r^{-\alpha/2}u\|_{L^2(\RR^n)}\|r^{1-\alpha/2} \del_r u\|_{L^2(\RR^n)}
		$$
		This leads to the desired result.
	\end{proof}
	
	Return to our energy density. Let us fix
	$$
	u_s(x) := u\big(\sqrt{s^2+r^2},x\big)
	$$
	being the restriction of $u$ on $\Hcal_s$. \\
	Then $\del_r u_s (x) = (x^a/r)\delu_a u(t,x)$ with $s = \sqrt{t^2-r^2}$.  Then by \eqref{eq1-30-04-2021} with $\alpha =1$, 
	\begin{equation}
		\aligned
		\int_{\Hcal_s}t|(x^a/r)\delu_a u|^2dx =& \int_{\Hcal_s}(t-r)|(x^a/r)\delu_a u|^2dx  + \int_{\RR_n}|r^{1-1/2}\del_ru_s|^2dx
		\\
		\geq& \int_{\Hcal_s}(t-r)|(x^a/r)\delu_a u|^2dx + \frac{(n-1)^2}{4}\|r^{-1/2}u_s\|_{L^2(\RR^n)}^2
		\\
		=& \int_{\Hcal_s}(t-r)|(x^a/r)\delu_a u|^2dx + \frac{(n-1)^2}{4}\int_{\Hcal_s} r^{-1}u^2dx
		\\
		=& \int_{\Hcal_s}(t-r)|(x^a/r)\delu_a u|^2dx + \frac{(n-1)^2}{4}\int_{\Hcal_s} \big(r^{-1}-t^{-1}\big)u^2dx
		\\
		& + \frac{(n-1)^2}{4}\int_{\Hcal_s} t^{-1}u^2dx  .
		\endaligned
	\end{equation}
	Combine this bound with \eqref{eq1-27-04-2021}, we obtain
	\begin{equation}\label{eq1-01-05-2021}
	\aligned
	\Ebf_1(s,u)
	:=&\int_{\Hcal_s}\ebf_1[u]dx 
	\\
	=& \frac{s^2}{2}\sum_{a<b}\|t^{-1/2}r^{-1}\Omega_{ab}u\|_{L^2(\Hcal_s)}^2 
	+ \int_{\Hcal_s} \frac{1}{2t}|K_1u|^2 + \frac{t}{2}|(x^a/r)\delu_a u|^2 + \frac{n-1}{2t}uK_1u\, dx
	\\
	\geq&\frac{1}{2}\sum_{a<b}\|(s/t)t^{1/2}r^{-1}\Omega_{ab}u\|_{L^2(\Hcal_s)}^2 
	+ \frac{1}{2}\|(t-r)^{1/2}(x^a/r)\delu_a u\|_{L^2(\Hcal_s)}^2
	\\
	&+ \frac{1}{2}\int_{\Hcal_s} t^{-1}(K_1u+(n-1)u/2)^2\,dx 
	+ \frac{(n-1)^2}{4}\Big\|\Big(\frac{t-r}{rt}\Big)^{1/2}u\Big\|_{L^2(\Hcal_s)}^2.
	\endaligned
	\end{equation}
	So we conclude by the following result:
	\begin{lemma}\label{lem1-02-05-2021}
		Let $u$ be sufficiently regular function defined in $\Kcal_{[s_0,s_1]}$ and vanishes near $\del\Kcal$. Then
		\\
		$\bullet$ when $n\geq 2$, the following quantities are bounded by $C\Ebf_1(s,u)^{1/2}$:
		\begin{equation}\label{eq2-01-05-2021}
			\aligned
			&\|t^{-1/2}(K_1u + (n-1)u/2)\|_{L^2(\Hcal_s)},
			\\
			&\|(s/t)t^{1/2}\delu_a u\|_{L^2(\Hcal_s)},\quad \|(s/t)^3t^{1/2}\del_{\alpha}u\|_{L^2(\Hcal_s)},\quad \|(s/t)t^{-1/2}u\|_{L^2(\Hcal_s)}.
			\endaligned
		\end{equation}
		\\
		$\bullet$ when $n=1$, the following quantities are bounded by $C\Ebf_1(s,u)^{1/2}$:
		\begin{equation}\label{eq3-01-05-2021}
			\|t^{-1/2}K_1u\|_{L^2(\Hcal_s)},\quad \|t^{1/2}\delu_xu\|_{L^2(\Hcal_s)}.
		\end{equation}
	\end{lemma}
	\begin{proof}
		The case $n=1$ is direct from \eqref{eq1-27-04-2021}. For the case $n\geq2$, we take \eqref{eq1-01-05-2021} and observe that, due to the fact $1\leq t-r \leq \frac{s^2}{t}$ in $\Kcal$, 
		$$
		\|(s/t)t^{-1/2}u\|_{L^2(\Hcal_s)} \leq C\Big\|\Big(\frac{t-r}{rt}\Big)^{1/2}u\Big\|_{L^2(\Hcal_s)}.
		$$
		This shows the bound on $u$. For the same reason, we obtain
		$$
		\|(s/t)t^{1/2}(x^a/r)\delu_a u\|_{L^2(\Hcal_s)}\leq C\|(t-r)^{1/2}(x^a/r)\delu_a u\|_{L^2(\Hcal_s)}
		$$
		Recall the bound on $\Omega u$ and the following identity
		$$
		\sum_{a=1}^n|\delu_a u|^2 = |(x^a/r)\delu_a u|^2 + \sum_{a<b}|r^{-1}\Omega_{ab}u|^2, 
		$$
		we obtain the bound on $\delu_a u$. 
		
		Finally, for $\del_{\alpha}u$, we only need to control $\del_t u$ because $\del_a u = \delu_au - (x^a/t)\del_tu$ where $\delu_a u$ is already bounded. For $\del_t u$, we remark the following identity:
		\begin{equation}\label{eq1-02-05-2021}
			t^{1/2}(s/t)^3\del_t u = (s/t)t^{-1/2}(K_1u + (n-1)u/2)  - (s/t)rt^{-1/2}(x^a/r)\delu_au - \frac{1}{2}(n-1)(s/t) t^{-1/2}u.
		\end{equation}
		The $L^2$ norm of right-hand-side is bounded by $C\Ebf_1(s,u)^{1/2}$.
	\end{proof}
	
	In \eqref{eq2-01-05-2021} neither $u$ nor $\del_{\alpha}u$ enjoys satisfactory bound. So we need the following technical result. 
	\begin{lemma}
		Let $u$ be a $C_1$ function defined in $\Kcal_{[s_0,s_1]}$ and vanishes near $\del\Kcal$. Then for $s\in [s_0,s_1]$ and $n\geq 1$,  
		\begin{equation}\label{eq9-28-04-2021}
			\|t^{-1/2}u\|_{L^2(\Hcal_s)} \leq \|t^{-1/2}u\|_{L^2(\Hcal_{s_0})} 
			+ \int_{s_0}^s{s^{\prime}}^{-1}\Ebf_1(s^{\prime},u)^{1/2}ds^{\prime}.
		\end{equation}
	\end{lemma}
	\begin{proof}
		We recall the bound \eqref{eq2-01-05-2021} on $(K_1u + (n-1)u/2)$ and write it in the following form:
		\begin{equation}\label{eq3-28-04-2021}
			\|t^{-n/2}K_1(t^{(n-1)/2}u)\|_{L^2(\Hcal_s)}\leq C\Ebf_1(s,u).
		\end{equation}
		
		On the other hand, let $w$ be a function defined in $\Kcal_{[s_0,s_1]}$, sufficiently regular and let
		$$
		v(s,x) := \frac{w\big(s\sqrt{1+|x|^2},sx\big)}{(1+|x|^2)^{n/4}} = (s/t)^{n/2}w\big|_{\big(s\sqrt{1+|x|^2},sx\big)}.
		$$
		Then one has the following relations:
		\begin{equation}\label{eq4-28-04-2021}
			\|v(s,\cdot)\|_{L^2(\RR^n)} = s^{-n/2}\|(s/t)^{n/2}w\|_{L^2(\Hcal_s)},\quad  \del_sv(s,x) = 
			s^{n/2-1}(t^{-n/2}K_1w)\Big|_{\big(s\sqrt{1+|x|^2},sx\big)}.
		\end{equation}
		Then we calculate
		$$
		\frac{d}{ds}\|v(s,\cdot)\|_{L^2(\RR^n)}^2 = 2\|v(s,\cdot)\|_{L^2(\RR^n)}\frac{d}{ds} \|v(s,\cdot)\|_{L^2(\RR^n)}.
		$$
		On the other hand
		$$
		\aligned
		\frac{d}{ds}\|v(s,\cdot)\|_{L^2(\RR^n)}^2
		=\frac{d}{ds}\int_{\RR^n}v^2(s,x)dx = 2\int_{\RR^n}v(s,x)\del_s v(s,x)dx
		\endaligned
		$$
		This leads to
		$$
		2\|v(s,\cdot)\|_{L^2(\RR^n)}\frac{d}{ds} \|v(s,\cdot)\|_{L^2(\RR^n)}
		\leq 2\|v(s,\cdot)\|_{L^2(\RR^n)}\|\del_s v(s,\cdot)\|_{L^2(\RR^n)}
		$$
		thus
		$$
		\frac{d}{ds} \|v(s,\cdot)\|_{L^2(\RR^n)}\leq \|\del_s v(s,\cdot)\|_{L^2(\RR^n)}.
		$$
		We concentrate on the right-hand-side:
		$$
		\aligned
		\|\del_s v(s,\cdot)\|_{L^2(\RR^n)}^2
		=& s^{n-2}\int_{\RR^n}\big|t^{-n/2}K_1w\big|^2_{\big(s\sqrt{1+|x|^2},sx\big)}dx 
		\\
		=& s^{-2}\int_{\RR^n}\big|t^{-n/2}K_1w\big|^2_{\big(\sqrt{s^2+|x^2|},x\big)}dx
		= s^{-2}\|t^{-n/2}K_1w\|_{L^2(\Hcal_s)}^2.
		\endaligned
		$$
		Then we obtain
		\begin{equation}\label{eq5-28-04-2021}
			\|v(s_1,\cdot)\|_{L^2(\RR^n)} \leq \|v(s_0,\cdot)\|_{L^2(\RR^n)} 
			+ \int_{s_0}^{s_1}s^{-1}\|t^{-n/2}K_1w\|_{L^2(\Hcal_s)}ds
		\end{equation}
		Now we take $w = t^{(n-1)/2}u$. Then $\|v(s,\cdot)\|_{L^2(\Hcal_s)} = \|t^{-1/2}u\|_{L^2(\Hcal_s)}$. Substitute these relations together with \eqref{eq3-28-04-2021} into \eqref{eq5-28-04-2021}, we obtain the desired result.
	\end{proof}
	For the convenience of expression, we introduce 
	$$
	\Fbf_1(s,u;s_0) := \Ebf_1(s,u) + \Ebf_1(s_0,u) + \int_{s_0}^s {s^{\prime}}^{-1} \Ebf_1(s^{\prime},u)\, ds^{\prime}.
	$$
	Here we remark that for a fixed $s_0$, $\|t^{-1/2}u\|_{L^2(\Hcal_{s_0})}\leq C(s_0)\|(s/t)t^{-1/2}u\|_{L^2(\Hcal_{s_0})}$. Then for some other component, we can also establish better bounds via $\Fbf_1(s,u;s_0)$:
	\begin{corollary}\label{cor1-02-05-2021}
		Let $u$ be a $C^1$ function defined in $\Kcal_{[s_0,s_1]}$ and vanishes near $\del\Kcal$. Then the following quantities are bounded by $C\Fbf_1(s,u;s_0)$:
		\begin{equation}
			\|t^{1/2}(x^a/r)\delu_a u\|_{L^2(\Hcal_s)},\quad \|t^{-1/2}K_1u\|_{L^2(\Hcal_s)}, \quad \|(s/t)^2t^{1/2}\del_{\alpha}u\|_{L^2(\Hcal_s)},\quad \|t^{-1/2}u\|_{L^2(\Hcal_s)}.
		\end{equation}
		Here $C$ is a constant determined by $n$.
	\end{corollary}
	\begin{proof}
		Let us recall \eqref{eq1-27-04-2021} in the following form:
		$$
		\aligned
		\ebf_1[u] =& \frac{1}{2t}|K_1u|^2 + \frac{t}{2}|(x^a/r)\delu_a u|^2 + \frac{s^2}{2t}\sum_{a<b}|r^{-1}\Omega_{ab} u|^2 + \frac{n-1}{2t}uK_1u
		\\
		=& \frac{1}{2t}\big(|K_1u|^2 + (n-1)uK_1u + (1/4)(n-1)^2u^2\big) + \frac{t}{2}|(x^a/r)\delu_a u|^2 + \frac{s^2}{2t}\sum_{a<b}|r^{-1}\Omega_{ab} u|^2
		\\
		& - \frac{(n-1)^2}{8t}u^2
		\\
		=& \frac{1}{2t}(K_1u + (n-1)u/2)^2 + \frac{t}{2}|(x^a/r)\delu_a u|^2 + \frac{s^2}{2t}\sum_{a<b}|r^{-1}\Omega_{ab} u|^2 - \frac{(n-1)^2}{8t}u^2.
		\endaligned
		$$
		Integrate the above identity, we obtain
		$$
		\aligned
		\Ebf_1(s,u) =& \frac{1}{2}\|t^{-1/2}(K_1u + (n-1)u/2)\|_{L^2(\Hcal_s)}^2 + \frac{1}{2}\|t^{1/2}(x^a/r)\delu_au\|_{L^2(\Hcal_s)}^2
		\\
		& + \frac{1}{2}\sum_{a<b}\|(s/t)t^{1/2}r^{-1}\Omega_{ab}u\|_{L^2(\Hcal_s)}^2 - \frac{(n-1)^2}{8}\|t^{-1/2}u\|_{L^2(\Hcal_s)}^2.
		\endaligned
		$$
		The last term is bounded by $C\Fbf_1(s,u;s_0)$. Then the bound on $(x^a/r)\delu_a u$ is established. Then we rewrite \eqref{eq1-02-05-2021} in the following form
		$$
		t^{1/2}(s/t)^2\del_t u = t^{-1/2}(K_1u + (n-1)u/2)  - rt^{-1/2}(x^a/r)\delu_au - \frac{1}{2}(n-1) t^{-1/2}u.
		$$
		The right-hand-side is bounded by $\Fbf_1(s,u;s_0)$. Then the bound on $\del_t u$ is established. For $\del_au$, we only need to recall that $\del_a u = \delu_au - (x^a/t)\del_t u$.
	\end{proof}
	
	\subsection{Energy estimate}
	Now we are ready to make the conclusion.
	\begin{theorem}\label{thm1-29-04-2021}
		Let $u$ be a $C^2$ function defined in $\Kcal_{[s_0,s_1]}$ and sufficiently regular. Then
		\begin{equation}\label{eq8-28-04-2021}
			\Ebf_1(s,u)^{1/2}\leq \Ebf_1(s_0,u)^{1/2} 
			+ C\int_{s_0}^s {s^{\prime}}^{1/2}\|(s^{\prime}/t)^{1/2}\Box u\|_{L^2(\Hcal_{s^{\prime}})}ds^{\prime}
		\end{equation}
		where $\Ebf_1(s,u) := \int_{\Hcal_s}\ebf_1[u]dx$. Furthermore, 
		\\
		$\bullet$ when $n\geq 2$, the following quantities and bounded by $C\Ebf_1(s,u)^{1/2}$:
		\begin{equation}\label{eq2-02-05-2021}
			\|(s/t)t^{1/2}\delu_a u\|_{L^2(\Hcal_s)},\quad \|(s/t)^3t^{1/2}\del_{\alpha}u\|, \quad \|(s/t)t^{-1/2}u\|_{L^2(\Hcal_s)}
		\end{equation}
		while the following quantity:
		\begin{equation}\label{eq3-02-05-2021}
			\|(s/t)^2t^{1/2}\del_{\alpha}u\|_{L^2(\Hcal_s)},\quad\|t^{1/2}(x^a/r)\delu_a u\|_{L^2(\Hcal_s)},\quad \|t^{-1/2}u\|_{L^2(\Hcal_s)},
		\end{equation}
		is bounded by $C\Fbf_1(s,u;s_0)$ with
		$$
		\Fbf_1(s,u;s_0) := \Ebf_1(s_0,u)^{1/2} + \Ebf_1(s,u)^{1/2} + \int_{s_0}^s{s^{\prime}}^{-1}\Ebf_1(s,u)^{1/2}ds^{\prime}.
		$$
		\\
		$\bullet$ when $n=1$, the following quantities are bounded by $C\Ebf_1(s,u)^{1/2}$:
		\begin{equation}\label{eq4-02-05-2021}
			\|t^{-1/2}K_1u\|_{L^2(\Hcal_s)},\quad \|t^{1/2}\delu_x u\|_{L^2(\Hcal_s)}
		\end{equation}
		while the following quantities are bounded by $C\Fbf_1(s,u;s_0)$:
		\begin{equation}
			\|(s/t)^2t^{1/2}\del_t u\|_{L^2(\Hcal_s)},\quad \|t^{-1/2}u\|_{L^2(\Hcal_s)}.
		\end{equation}
		The above constant $C$ is determined by $n$.
	\end{theorem}
	\begin{proof}
		For \eqref{eq8-28-04-2021}, we integrate \eqref{eq7-28-04-2021} in $\Kcal_{[s_0,s]}$ and apply Stokes formula. Classical calculation leads us to
		$$
		\aligned
		\Ebf_1(s,u) - \Ebf_1(s_0,u) =& \int_{\Kcal_{[s_0,s]}} (K_1u + (n-1)u/2)\Box u dxdt
		\\
		=& \int_{s_0}^s\int_{\Hcal_s^{\prime}} (K_1u + (n-1)u/2)\Box u\, (s^{\prime}/t)dxds^{\prime}.
		\endaligned
		$$
		Differentiate the above identity with respect to $s$, we obtain
		$$
		\frac{d}{ds}\Ebf_1(s,u) = \int_{\Hcal_s}(s/t)(K_1u + (n-1)u/2)\Box u\, dx. 
		$$
		Then apply the bounds in \eqref{eq2-01-05-2021} and \eqref{eq3-01-05-2021} for the case $n=1$ and $n\geq 2$,
		$$
		\Ebf_1(s,u)^{1/2}\frac{d}{ds}\Ebf_1(s,u)^{1/2}
		\leq \|(s/t)t^{1/2}\Box u\|_{L^2(\Hcal_s)}\|t^{-1/2}(K_1u + (n-1)u/2)\|_{L^2(\Hcal_s)}
		$$
		which leads to
		$$
		\frac{d}{ds}\Ebf_1(s,u)^{1/2}\leq Cs^{1/2}\|(s/t)^{1/2}\Box u\|_{L^2(\Hcal_s)}.
		$$
		Integrate the above inequality, we obtain \eqref{eq8-28-04-2021}.
		
		For the bounds on the terms in lists \eqref{eq1-02-05-2021}, \eqref{eq3-02-05-2021} and \eqref{eq4-02-05-2021}, we only need to recall Lemma \ref{lem1-02-05-2021} and Corollary \ref{cor1-02-05-2021}.
	\end{proof}
	

	\section{Sharp decay estimate on Klein-Gordon equation}\label{sec4-13-05-2021}
	
Form this section we fix $n=3$. We also apply $A\lesssim B$ for $A\leq CB$ with $C$ an irrelevant constant.
	\subsection{Objective}
	In this section we recall and reformulate the technique introduced in \cite{Klainerman85} and applied in \cite{FangD} combined with normal form method. The main observation is that, based on the hyperbolic decomposition of wave operator, one can reduce Klein-Gordon equation into an ODE (more precisely, a harmonic oscillator). Then the techniques on ODE will be applied in order to obtain decay estimates which are more ``robust'' in the sens that it can undertake considerable perturbations.   
	\subsection{Differential identity}

	
	\begin{proposition}\label{prop1-11-05-2021}
		Let $v$ be a sufficiently regular solution to
		\begin{equation}\label{starting KG-sharp}
			g^{\alpha\beta}\del_{\alpha}\del_{\beta} v + c^2(1 + P)v = f
		\end{equation}
		in $\Hcal_{[s_0,s_1]}$ and vanishing near $\del \Kcal$.  $g^{\alpha\beta} = m^{\alpha\beta} + H^{\alpha\beta}$ is a sufficiently regular metric defined in $\Hcal_{[2,s_1]}$. We denote by
		$\Hb := (s/t)^{-2}\Hu^{00}$ and suppose that
		\begin{equation}\label{eq4-11-05-2021}
			|\Hb|\leq 1/2, \quad |H|_{p,k}\leq 1/2.
		\end{equation}
		$P, f$ are sufficiently regular functions defined in $\Hcal_{[2,s_1]}$ with $|P|\leq 1/2$, $|P|_{p,k}\leq C$. Then
		\begin{equation}\label{eq1-10-05-2021}
			\Lcal^2(s^{3/2}v) + c^2(1+P-\Hb)s^{3/2}v = s^{3/2}S_2[v] + \frac{s^{3/2}f}{1+\Hb}
		\end{equation}
		where $\Lcal = (s/t)\del_t + (x^a/s)\delu_a$ and
		\begin{equation}\label{eq2-06-05-2021}
			\aligned
			&s^{3/2}|S_2[\del^IL^J v]|\lesssim s^{-1/2}|v|_{p+2}  + (s/t)^{-1}s^{1/2}|H||\del v|_{p+1} + |\Hb||P||v|_p,\quad |I|+|J|\leq p,\quad |J|\leq k,
			\\
			&s^{3/2}|(1+\Hb)f|_{p,k}\lesssim |f|_{p,k}.
			\endaligned
		\end{equation}
	\end{proposition}
	\begin{proof}
We first recall the hyperbolic parameterization $(s,x)$ of $\Kcal$. 
$$
s = \sqrt{t^2-|x|^2},\quad x^a = x^a.
$$
The associate the natural frame is
$$
\delb_0=\delb_s = (s/t)\del_t,\quad \delb_a = \delu_a = (x^a/t)\del_t + \del_a. 
$$
It is evident that $[\delb_{\alpha},\delb_{\beta}] = 0$.

		Then we preform the following calculation:
		\begin{equation}\label{KG-sharp-1}
			\Box v + c^2(1+P)v + (s/t)^{-2}\Hu^{00}\delb_s\delb_sv = f - \Hu(\del\del v) + \Hu^{00}(s/t)^{-1}s^{-1}(r/t)^2\del_t v
		\end{equation}
		where
		$$
		\Hu(\del\del v) = \sum_{(\alpha,\beta)\neq (0,0)}\Hu^{\alpha\beta}\delu_{\alpha}\delu_{\beta}u + H^{\alpha\beta}\del_{\alpha}\big(\Psiu_{\beta}^{\beta'}\big)\delu_{\beta'}v.
		$$
		Then we recall the following hyperbolic decomposition of D'Alembert operator:
		\begin{equation}\label{KG-sharp-2}
			\aligned
			\Box v =& \delb_s\delb_sv + \frac{2x^a}{s}\delb_s\delb_av - \sum_a\delb_a\delb_av + \frac{3}{s}\delb_sv. 
			\endaligned
		\end{equation}
		Then combine \eqref{KG-sharp-1} and \eqref{KG-sharp-2}, we obtain:
		$$
		\aligned
		\big(1+(s/t)^{-2}\Hu^{00}\big)\delb_s\delb_sv &+ \frac{2x^a}{s}\delb_s\delb_av + \frac{3}{s}\delb_sv + c^2(1+P)v 
		\\
		=& f - \Hu(\del\del v) + \Hu^{00}(s/t)^{-1}s^{-1}(r/t)^2\del_t v + \sum_a\delb_a\delb_av 
		\endaligned
		$$ 
		then, denote by $\bar{H} = (s/t)^{-2}\Hu^{00}$,
		\begin{equation}\label{KG-sharp-3}
			\aligned
			\delb_s\delb_sv + \frac{2x^a}{s}\delb_s\delb_av + \frac{3}{s}\delb_sv  + \frac{c^2(1+P)v}{1+\Hb}
			= S_1[H,P,v] + \frac{f}{1+\Hb}
			\endaligned
		\end{equation}
		where
		$$
		\aligned
		S_1[H,P,v]
		=&\big(1-(1+\Hb)^{-1}\big)\Big(\frac{2x^a}{s}\delb_s\delb_av + \frac{2}{s}\delb_sv\Big) 
		\\
		&+(1+\Hb)^{-1}\Big(\Hu^{00}(s/t)^{-1}s^{-1}(r/t)^2\del_t v -  \Hu(\del\del v) + \sum_a\delb_a\delb_av\Big) .
		\endaligned
		$$
		Now remark the following identity:
		$$
		\delb_s\delb_sv + \frac{2x^a}{s}\delb_s\delb_av + \frac{3}{s}\delb_sv  = s^{-3/2}\big(\delb_s+(x^a/s)\delb_a\big)^2(s^{3/2}v) - \frac{x^ax^b}{s^2}\delb_a\delb_bv 
		- \frac{3x^a}{s^2}\delb_a v - \frac{3}{4s^2}v.
		$$
		Then combined with \eqref{KG-sharp-3} and denote by $\Lcal := \delb_s + (x^a/s)\delb_a$, we obtain \eqref{eq1-10-05-2021} with
		$$
		\aligned
		S_2[H,P,v] =& s^{-2}(x^ax^b\delb_a\delb_b v + 3x^a\delu_a + 3v/4) + S_1[H,P,v]
		\\
		=&s^{-2}(x^ax^b\delb_a\delb_b v + 3x^a\delu_a + 3v/4)
		+\big(1-(1+\Hb)^{-1} - (1+P)^{-1}\Hb\big)c^2(1+P)v
		\\
		&+ \big(1-(1+\Hb)^{-1}\big)\Big(\frac{2x^a}{s}\delb_s\delb_av + \frac{2}{s}\delb_sv\Big)
		\\
		&+(1+\Hb)^{-1}\Big(\Hu^{00}(s/t)^{-1}s^{-1}(r/t)^2\del_t v -  \Hu(\del\del v) + \sum_a\delb_a\delb_av\Big).
		\endaligned
		$$
		
		We now concentrate on the bounds on the bound of $S_2$. Firstly we need to control $(1+\Hb)^{-1}, 1-(1+\Hb)^{-1}$ and $1-\Hb - (1+\Hb)^{-1}$. When $|\Hb|\leq 1/2$,
		$$
		(1+\Hb)^{-1} = \sum_{k=0}^{\infty} (-\Hb)^k,\quad 1-(1+\Hb)^{-1} = -\sum_{k=1}^{\infty}(-\Hb)^k,\quad 1-\Hb - (1+\Hb)^{-1}= -\sum_{k=2}^{\infty}(-\Hb)^k.
		$$
		When $|\Hb|_{[p/2]}\leq 1/2$, one may differentiate the above identities term by term. Then we obtain
		\begin{equation}\label{eq1-06-05-2021}
			\aligned
			|(1+\Hb)^{-1}|_p\lesssim& 1,\quad |(1+\Hb)|_p\lesssim |\Hb|_p,\quad \text{for } p\geq 1,
			\\
			|1-(1+\Hb)^{-1}|_p\lesssim& |\Hb|_p,
			\\
			|1-\Hb-(1+\Hb)^{-1}|_p\lesssim& |\Hb|_p|\Hb|_{[p/2]}. 
			\endaligned
		\end{equation}
		Equipped with the above bounds, it is direct to see the following bound on $S_2[H,P,v]$:
		$$
		|S_2[H,P,\del^IL^Jv]|\lesssim s^{-2}|v|_{p+2}  + (s/t)^{-1}s^{-1}|H||\del v|_{p+1}
	+ |\Hb||P||v|_p,\quad |I|+|J|\leq p,\quad |J|\leq k.
		$$
		In the same manner,
		$$
		|s^{3/2}(1+\Hb)^{-1}f|_{p,k}\lesssim s^{3/2}|f|_{p,k}.
		$$
		The we established the desired bound.
	\end{proof}
	
	\subsection{The estimate}
	We firstly recall an ODE result.
	\begin{lemma}\label{lem1-05-11-2021}
		Let $v$ be a $C^2$ function defined on $[s_0,s_1]$ and satisfies the ODE
		$$
		v''(s) + c^2(1+q(s)) v(s) = f(s)
		$$
		with $q,f$ sufficiently regular functions defined on $[s_0,s_1]$ and $|q|\leq 1/2$. Then 
		\begin{equation}\label{eq2-11-05-2021}
			|v'(s)| + c|v(s)|\leq |v'(s_0)| + c|v(s_0)| + Cc^{-1}\int_{s_0}^s|f(s^{\prime})| + |q'(s)v'(s)|\, ds^{\prime}.
		\end{equation}
	\end{lemma}
	
	Then we establish the following estimate.
	\begin{proposition}\label{prop2-11-05-2021}
		Let $v$ be a sufficiently regular solution to \eqref{starting KG-sharp} with \eqref{eq4-11-05-2021}. Then when $s_0\geq 2$,
		\begin{equation}\label{eq5-11-05-2021}
			\aligned
			s^{3/2}\big(|\Lcal v(s)| + |v(s)|\big) &\lesssim s_0^{3/2}\sup_{\Hinit_{s_0}}\big(|\Lcal v(s)| + |v(s)|\big)
			+ c^{-1}\int_{\lambda_0}^s\Big|s^{3/2}S_2[H,P,v] + \frac{s^{3/2}f}{1+\Hb})\Big|_{\gamma_{t,x}(\lambda)}\, d\lambda 
			\\
			&+ c^{-1}\int_{\lambda_0}^s \Big|\Lcal(\Hb-P)(s^{3/2}\Lcal v+(3/2)s^{1/2}v)\Big|_{\gamma_{t,x}(\lambda)}d\lambda
			\endaligned
		\end{equation}
		with
		\begin{equation}\label{eq6-11-05-2021}
			\lambda_0 = 
			\left\{
			\aligned
			&s_0,\quad && 0\leq |x|/t\leq \frac{s_0^2-1}{s_0^2+1},
			\\
			&\sqrt{\frac{t+r}{t-r}},&& \frac{s_0^2-1}{s_0^2+1}\leq |x|/t<1.
			\endaligned
			\right.
		\end{equation}
	\end{proposition}
	\begin{remark}
		It is important to remark that when $|x|/t\geq \frac{s_0^2-1}{s_0^2+1}$, $\gamma_{t,x}(\lambda_0)\in \Kcal_{[s_0,s_1]}$. Due to the fact that $v$ vanishes near $\del\Kcal$, one has the following bound ``near light cone''
		\begin{equation}
			\aligned
			s^{3/2}\big(|\Lcal v(s)| + |v(s)|\big) &\lesssim  c^{-1}\int_{\lambda_0}^s\lambda^{3/2}\Big|S_2[H,P,v] + \frac{f}{1+\Hb})\Big|_{\gamma_{t,x}(\lambda)}\, d\lambda 
			\\
			&+ c^{-1}\int_{\lambda_0}^s \Big|\Lcal(\Hb-P)(s^{3/2}\Lcal v+(3/2)s^{1/2}v)\Big|_{\gamma_{t,x}(\lambda)}d\lambda
			\endaligned
		\end{equation}
		with $\lambda_0 = \sqrt{\frac{t+r}{t-r}}\simeq (t/s)$. 
	\end{remark}
	\begin{remark}
		Recall that $\Lcal v = \delb_s v + (x^a/s)\delu_a v = (s/t)\del_t v + s^{-1}(x^a/t)L_a v$. Remark that $\del_a v = t^{-1}L_av - (x^a/t)\del_t v$
		Then
		\begin{equation}\label{eq8-11-05-2021}
			(s/t)|\del_{\alpha}v|\lesssim|\Lcal v| + s^{-1}|v|_{1,1}.
		\end{equation}
	\end{remark}
	
	\begin{proof}
		We need to write \eqref{eq1-10-05-2021} into the an ODE from. Let 
		$$
		\aligned
		\gamma_{t,x}: [s_0,s_1]\rightarrow& \Hcal_{[s_0,s_1]},
		\\
		\lambda \rightarrow& (\lambda t/s,\lambda x/s).
		\endaligned
		$$
		We firstly remark that there exists a $s_1\geq \lambda_0\geq s_0$, such that $\gamma_{t,x}(\lambda_0)\in \Hinit_{s_0}$. The explicit value of $\lambda_0$ is given by \eqref{eq6-11-05-2021}. Now let 
		$$
		w_{t,x} (\lambda) := s^{3/2}v\big|_{\gamma_{t,x}(\lambda)} = \lambda^{3/2} v(\lambda t/s,\lambda x/s),\quad \text{ then}\quad
		w_{t,x}'(\lambda) := \Lcal(s^{3/2}v)\big|_{\gamma_{t,x}(\lambda)}. 
		$$
		Then \eqref{eq1-10-05-2021} is written as
		\begin{equation}\label{eq3-11-05-2021}
			w_{t,x}''(\lambda) + c^2(1 + P - \Hb)w_{t,x}(\lambda) = s^{3/2}S_1[v]\big|_{\gamma_{t,x}(\lambda)} + \frac{s^{3/2}f}{1+\Hb}\Big|_{\gamma_{t,x}(\lambda)}.
		\end{equation}
		By Lemma \ref{lem1-05-11-2021} with $w_{t,x} = v$, $(\Hb-P)\big|_{\gamma_{t,x}} = q$ and $s^{3/2}S_1[v]\big|_{\gamma_{t,x}} + \frac{s^{3/2}f}{1+\Hb}\Big|_{\gamma_{t,x}} = f$,  we obtain the following estimate:
		\begin{equation}\label{eq7-11-05-2021}
			\aligned
			|w_{t,x}'(\lambda)| + c|w_{t,x}(\lambda)|\leq& |\Lcal w_{t,x}(\lambda_0)| + c|w_{t,x}(\lambda_0)| 
			+ C\int_{\lambda_0}^\lambda \Big|s^{3/2}S_1[v] + \frac{s^{3/2}f}{1+\Hb}\Big|_{\gamma_{t,x}(\lambda')}d\lambda'
			\\
			&+ C\int_{\lambda_0}^\lambda\Big|\Lcal(\Hb-P)(s^{3/2}\Lcal v + (3/2)s^{1/2}v)\Big|_{\gamma_{t,x}(\lambda')}d\lambda'
			\endaligned
		\end{equation}
		where we remark that
		$$
		\aligned
		w_{t,x}'(\lambda) =& \Lcal(s^{3/2}v)\big|_{\gamma_{t,x}(\lambda)} = (3/2)s^{1/2}v\big|_{\gamma_{t,x}(\lambda)} + s^{3/2}\Lcal v\big|_{\gamma_{t,x}(\lambda)}.
		\endaligned
		$$
		Then in \eqref{eq7-11-05-2021} we fix $\lambda = s$, the desired result is established.
	\end{proof}
	
	\section{Other technical tools}\label{sec3-13-05-2021}
	\subsection{Standard energy estimate}
	We recall the following standard energy defined on hyperboloids. They are obtained by applying the multiplier $\del_t u$ on $\Box u + c^2u$.
	\begin{equation}\label{eq8-19-08-2022-M}
		\aligned
		\Ebf_{0,c}(s,u):=& \int_{\Hcal_s}\Big((\del_tu)^2 + \sum_a(\del_au)^2 + 2(x^a/t)\del_tu\del_au + c^2u^2\Big) dx
		\\
		=&\int_{\Hcal_s} \Big(\big((s/t)\del_t u\big)^2 + \sum_a (\delu_a u)^2  + c^2u^2\Big) \, dx
		\\
		=&\int_{\Hcal_s}\left((\delu_{\perp} u)^2 + \sum_a\left((s/t)\del_a u\right)^2 + \sum_{a<b}\left(t^{-1}\Omega_{ab}u\right)^2 + c^2u^2 \right) \, dx,
		\endaligned
	\end{equation}
	
	and, for $g^{\alpha\beta}\del_{\alpha}\del_{\beta}u + c^2u = f$ with $g^{\alpha\beta} = H^{\alpha\beta} + m^{\alpha\beta}$,
	\begin{equation}\label{eq23-12-05-2021}
		\aligned
		\Ebf_{0,g,c}(s,u) =& \int_{\Hcal_s} g^{00}|\del_tu|^2 - (2x^a/t)g^{ab}\del_au\del_bu + c^2u^2\, dx
		\\
		=& \Ebf_{0,c}(s,u) + \int_{\Hcal_s} H^{00}|\del_tu|^2 - (2x^a/t)H^{ab}\del_au\del_bu \, dx.
		\endaligned
	\end{equation}
	We state the following standard energy estimate for quasi-linear wave / Klein-Gordon equation on hyperboloids. A detailed proof can be found in \cite{PLF-YM-book1}. 
	\begin{proposition}
		\label{prop2-12-05-2021}
		1.
		For every $C^2$ function $u$ which is defined in the region $\Kcal_{[2,s]}$ and vanishes near $\del \Kcal$,
		one has for all $s \geq 2$
		\begin{equation}\label{eq17-12-05-2021}
			\Ebf_0(s,u)^{1/2}\leq \Ebf_0(2,u)^{1/2} + \int_2^s \| \Box u\|_{L^2(\Hcal_{s^{\prime}})} d s^{\prime}.
		\end{equation}
		
		2. Let $v$ be a $C^2$ solution to the Klein-Gordon equation on a curved space time
		$$
		g^{\alpha\beta}\del_{\alpha}\del_{\beta} v + c^2 v = f,
		$$
		defined the region $\Kcal_{[2,s]}$ and vanishes near $\del\Kcal$. Suppose that $H^{\alpha\beta}
		= g^{\alpha\beta} - m^{\alpha\beta}$ satisfies the following two conditions
		(for some constant $\kappa \geq 1$ and some function $M$):
		\begin{subequations}\label{eq22-12-05-2021}
			\begin{equation}\label{eq18-12-05-2021}
				\kappa^{-2} \Ebf_{0,g,c}(s,v) \leq \Ebf_{0,c}(s,v) \leq \kappa^2 \Ebf_{0,g,c}(s,v),
			\end{equation}
			\begin{equation}\label{eq19-12-05-2021}
				\bigg|\int_{\Hcal_s}(s/t)\Big(\del_\alpha H^{\alpha\beta}
				\del_t v\del_\beta v - \frac{1}{2}\del_t H^{\alpha\beta}\del_\alpha v\del_\beta v\Big) \, dx\bigg|
				\leq M(s)\Ebf_{0,c}(s,v)^{1/2}.
			\end{equation}
		\end{subequations}
		Then, the evolution of the hyperboloidal energy is controlled (for all $s \geq 2$) by
		\begin{equation}\label{eq20-12-05-2021}
			\Ebf_{0,c}(s,v)^{1/2}\leq \kappa^2 \Ebf_{0,c}(2,v)^{1/2} + \kappa^2 \int_2^s \Big(\|f\|_{L^2(\Hcal_{s^{\prime}})} + M(s^{\prime})\Big) d s^{\prime}.
		\end{equation}
	\end{proposition}
	
	\subsection{Controlling high-order derivatives and Sobolev decay}
	For the convenience of discussion, we introduce the following high-order energies:
	$$
	\aligned
	&\Ebf_{0,c}^{p,k}(s,u) := \sum_{|I|+|J|\leq p\atop |J|\leq k}\Ebf_{0,c}(s,\del^IL^J u),\quad &&\Ebf_0^{p,k}(s,u) := \sum_{|I|+|J|\leq p\atop |J|\leq k}\Ebf_0(s,\del^IL^J u),
	\\
	&\Ebf_{0,c}^N(s,u) := \sum_{|I|+|J|\leq N}\Ebf_{0,c}(s,\del^IL^J u),\quad &&\Ebf_0^N(s,u) := \sum_{|I|+|J|\leq N}\Ebf_0(s,\del^IL^J u),
	\\
	&\Ebf_1^{p,k}(s,u) := \sum_{|I|+|J|\leq p\atop |J|\leq k}\Ebf_1(s,\del^IL^J u),\quad &&\Ebf_1^N(s,u) := \sum_{|I|+ |J|\leq N}\Ebf_1(s,\del^IL^J u),
	\\
	&\Fbf_1^{p,k}(s,u;s_0) :=\sum_{|I|+|J|\leq p\atop |J|\leq k}\Fbf_1(s,\del^IL^J u;s_0),\quad &&\Fbf_1^N(s,u;s_0) := \sum_{|I|+ |J|\leq N}\Fbf_1(s,\del^IL^J u;s_0).
	\endaligned
	$$
	
	We firstly recall the bounds with standard energy established in \cite{PLF-YM-book1} (see \cite{M-2021-strong} for a brief proof):
	\begin{equation}\label{eq1-10-06-2020}
		\|(s/t)|\del u|_{p,k}\|_{L^2(\Hcal_s)} + \||\dels u|_{p,k}\|_{L^2(\Hcal_s)} 
		+ \|c|u|_{p,k}\|_{L^2(\Hcal_s)}\leq C\Ebf_{0,c}^{p,k}(s,u)^{1/2},
	\end{equation}
	\begin{equation}\label{eq5-10-06-2020}
		\|s|\del\dels u|_{p-1,k-1}\|_{L^2(\Hcal_s)} + \|t|\dels\dels u|_{p-1,k-1}\|_{L^2(\Hcal_s)} \leq C\Ebf_{0,c}^{p,k}(s,u)^{1/2},
	\end{equation}
	\begin{equation}\label{eq3-10-06-2020}
		\aligned
		\|(s/t)t^{3/2}|\del u|_{p,k}\|_{L^{\infty}(\Hcal_s)}& + \|t^{3/2}|\dels u|_{p,k}\|_{L^\infty(\Hcal_s)}  + \|ct^{3/2}|u|_{p,k}\|_{L^{\infty}(\Hcal_s)}
		\leq C\Ebf_{0,c}^{p+2,k+2}(s,u)^{1/2},
		\endaligned
	\end{equation}
	\begin{equation}\label{eq7-10-06-2020}
		\|st^{3/2}|\del\dels u|_{p-1,k-1}\|_{L^\infty(\Hcal_s)} + \|t^{5/2}|\dels\dels u|_{p-1,k-1}\|_{L^{\infty}(\Hcal_s)} \leq C\Ebf_{0,c}^{p+2,k+2}(s,u)^{1/2}.
	\end{equation}
	Then we establish parallel bounds with $\Ebf_1$ energies:
	\begin{lemma}\label{lem1-14-05-2021}
		Let $u$ be a sufficiently regular function defined in $\Hcal_{[s_0,s_1]}$. Then
		\begin{equation}\label{eq4-14-05-2021}
			\|(s/t)t^{1/2}|\dels u|_{p,k}\|_{L^2(\Hcal_s)},\quad \|(s/t)^3t^{1/2}|\del u|_{p,k}\|_{L^2(\Hcal_s)},\quad \|(s/t)t^{-1/2}|u|_{p,k}\|_{L^2(\Hcal_s)}
		\end{equation}
		are bounded by $C(p)\Ebf_1^{p,k}(s,u)^{1/2}$ with $C$ determined by $p$, and 
		\begin{equation}\label{eq5-14-05-2021}
			\|(s/t)^2t^{1/2} |\del u|_{p,k}\|_{L^2(\Hcal_s)},\quad \|t^{-1/2}|u|_{p,k}\|_{L^2(\Hcal_s)},\quad \|t^{1/2}|\dels u|_{p-1,k-1}\|_{L^2(\Hcal_s)}
		\end{equation}
		are bounded by $C(p)\Fbf_1^{p,k}(s,u;s_0)^{1/2}$. 
		
		Furthermore,
		\begin{equation}\label{eq6-14-05-2021}
			\|(s/t)t^2|\dels u|_{p,k}\|_{L^\infty(\Hcal_s)},\quad \|(s/t)^3t^2|\del u|_{p,k}\|_{L^\infty(\Hcal_s)}
		\end{equation}
		are bounded by $C(p)\Ebf_1^{p+2,k+2}(s,u;s_0)$ and 
		\begin{equation}\label{eq7-14-05-2021}
			\|(s/t)^2t^2|\del u|_{p,k}\|_{L^2(\Hcal_s)},\quad \|t|u|_{p,k}\|_{L^2(\Hcal_s)},\quad 
			\|t^2|\dels u|_{p-1,k-1}\|_{L^2(\Hcal_s)}
		\end{equation}
		are bounded by  $C(p)\Fbf_1^{p+2,k+2}(s,u;s_0)$.
	\end{lemma}
	\begin{proof}
		These are established in the same manner. We only need to recall the following inequality established in \cite{PLF-YM-book1}(see \cite{M-2021-strong} for a brief proof)):
		\begin{equation}\label{eq11-10-06-2020}
			\aligned
			|\del u|_{p,k}\leq C\sum_{|I|+|J|\leq p\atop |J|\leq k,\alpha}|\del_{\alpha}\del^IL^Ju|,
			\quad 
			|(s/t)\del u|_{p,k}\leq C(s/t)\sum_{|I|+|J|\leq p\atop |J|\leq k,\alpha}|\del_{\alpha}\del^IL^Ju|,
			\endaligned
		\end{equation} 
		\begin{equation}\label{eq3 notation}
			|\dels u|_{p,k}\leq C\sum_{|I|+|J|\leq p,a\atop |J|\leq k}|\delu_a\del^IL^Ju|
			+ Ct^{-1}\sum_{ |J|\leq k,\alpha\atop 0\leq |I|+|J|\leq p-1}|\del_{\alpha}\del^IL^Ju|\leq Ct^{-1}|u|_{p+1,k+1}.
		\end{equation}
		Then regarding \eqref{eq2-02-05-2021} and \eqref{eq3-02-05-2021} together with the definition of high-order energies, the bounds on $L^2$ norms are direct. For the bounds on $L^{\infty}$ norms, we need to recall the following Klainermain-Sobolev type inequality on hyperboloids:
		$$
		t^{3/2}|u(t,x)| \leq C\sum_{|I|+|J|\leq 2}\|\del^IL^J u\|_{L^2(\Hcal_s)},\quad (t,x)\in\Hcal_s
		$$
		where $u$ is a $C^2$ function defined in $\Kcal_{[s_0,s_1]}$ vanishes near light-cone $\del\Kcal = \{r=t-1\}$. Then we regard for example for $K \in \mathcal{I}_{p,k}$ and $l,j\in\RR$,6
		$$
		\sum_{|I|+|J|\leq 2}|\del^IL^J \big((s/t)^lt^jZ^K \del_{\alpha}u\big)|\leq C\sum_{|I'|+|J'|\leq p+2\atop |J'|\leq k+2,\beta}|(s/t)^lt^j\del_{\beta}\del^{I'}L^{J'}u|.
		$$
		This leads to $s|Z^K\del_{\alpha}u(t,x)|\leq C\Ebf_0^{p+2,k+2}(s,u)^{1/2}$. Here we have applied the relations
		$$
		|\del^IL^J(s/t)|\leq C(s/t),\quad  |\del^IL^J t|\leq Ct^{1-|I|}
		$$
		in $\Kcal$. This can be proved easily by induction on $|I|$ and $|J|$ (see \cite{M-2021-strong} for a detailed proof).
	\end{proof}
	
	\subsection{Sharp decay estimate on wave equation}
	We recall the following bound based on Kirchhoff's formula. See \cite{PLF-YM-CMP} for a detailed proof.
	\begin{proposition}\label{prop1-12-05-2021}
		Let $u$ be a $C^2$ solution to the following Cauchy problem
		\begin{equation}
			\Box u = f,\quad u(2,x) = \del_t u(2,x)=0.
		\end{equation}
		with
		$$
		|f|\leq C_F \mathbbm{1}_{\{|x|\leq t-1\}}t^{-2-\nu}(t-r)^{-1+\mu}
		$$
		for some constant $C_F\geq 0$ and $0<\mu,|\nu|\leq 1/2$. Then
		Then
		\begin{equation}\label{eq8-12-05-2021}
			|u(t,x)|\leq
			\left\{
			\aligned
			&CC_F\mu^{-1}|\nu|^{-1} (t-r)^{\mu-\nu}t^{-1},\quad && 0<\nu\leq 1/2,
			\\
			&CC_F\mu^{-1}|\nu|^{-1} (t-r)^{-\mu}t^{-1-\nu},\quad && -1/2\leq \nu<0.
			\endaligned
			\right.
		\end{equation}
	\end{proposition}

	\section{Global existence }\label{sec5-13-05-2021}
	\subsection{Bootstrap assumption and direct consequences}
	\paragraph{Bootstrap assumptions.}
	Regarding the local theory, we can propagate the local solution to the initial hyperboloid $\Hcal_2 = \{t = \sqrt{4+r^2}\}$ given that $\vep$ sufficiently small. Further, the restriction of the local solution satisfies the following energy estimate:
	\begin{equation}\label{eq25-12-05-2021}
		\Ebf_1^{N}(2,u)^{1/2} + 4\Ebf_{0,c}^N(2,v)^{1/2}\leq C_0\vep
	\end{equation}
	where $C_0$ is a constant determined by the system, $N$ and $\|u_{\ell}\|_{H^{N+1-\ell}(\RR^3)},\|v_{\ell}\|_{H^{N+1-\ell}(\RR^3)}$. Then we follow the standard bootstrap argument. Suppose that on each time interval $[2,s_1]$ one has 
	\begin{equation}\label{eq12-11-05-2021}
		\Ebf_1^N(s,u)^{1/2} + 4\Ebf_{0,c}^N(s,v)^{1/2} \leq C_1\eps s^{\delta}.
	\end{equation}
	with $C_1\geq C_0$ and $0<\delta\leq 1/20$. If we can show that on the same interval the following {\it improved energy estimates} hold:
	\begin{equation}
		\Ebf_1^N(s,u)^{1/2} + 4\Ebf_{0,c}^N(s,v)^{1/2} \leq \frac{1}{2}C_1\eps s^{\delta}.	
	\end{equation}
	Then by the standard bootstrap argument, we conclude that the local solution extends to time infinity. Furthermore, \eqref{eq12-11-05-2021} is valid globally in tome.
	
	From now on the constant $C$ may depend on the coefficients of the system, i.e., $P^{\alpha\beta}$, $B^{\alpha\beta},R$ and $\delta^{-1}$.

	\paragraph{Sobolev bounds.} We apply Lemma~\ref{lem1-14-05-2021} together with the bootstrap assumptions \eqref{eq12-11-05-2021}. For wave component:
	\begin{subequations}\label{eq2-13-05-2021}
		\begin{equation}\label{eq10-11-05-2021}
			|\del u|_{N-2}\leq CC_1\vep s^{-2+\delta},\quad |\dels u|_{N-2}\leq CC_1\vep (s/t)s^{-2+\delta},
		\end{equation}
		\begin{equation}\label{eq11-11-05-2021}
			|u|_{N-2}\leq CC_1\vep (s/t)s^{-1 + \delta},\quad |\dels u|_{N-3}\leq C C_1\vep (s/t)^2s^{-2+\delta}.
		\end{equation}
	\end{subequations}
	For Klein-Gordon component:
	\begin{subequations}\label{eq3-13-05-2021}
		\begin{equation}\label{eq3a-13-05-2021}
			|\del v|_{N-2}\leq CC_1\vep (s/t)^{1/2}s^{-3/2+\delta},\quad |\dels v|_{N-2}\leq CC_1\vep (s/t)^{3/2}s^{-3/2+\delta},
		\end{equation}
		\begin{equation}\label{eq3b-13-05-2021}
			|v|_{N-2} + |\del v|_{N-3}\leq CC_1\vep (s/t)^{3/2}s^{-3/2+\delta},
			\quad |\dels v|_{N-3}\leq CC_1\vep (s/t)^{5/2}s^{-5/2+\delta}.
		\end{equation}
	\end{subequations}
	\subsection{Uniform standard energy bound of wave component and related Sobolev decay}
	We establish the following uniform standard energy bound on wave component:
	\begin{equation}\label{eq1-13-05-2021}
		\Ebf_0^N(s,u)^{1/2}\leq C_0\vep + C(C_1\vep)^2\leq CC_1\vep.
	\end{equation}
	The above estimate leads to the following results: by \eqref{eq3-10-06-2020}, we obtain 
	\begin{equation}\label{eq14-12-05-2021}
		|\del u|_{N-3}\leq CC_1\vep (s/t)^{1/2}s^{-3/2},
	\end{equation}
	\begin{equation}\label{eq1-12-05-2021}
		|u|_{N-3}\leq CC_1\vep (s/t)^{3/2}s^{-1/2}.
	\end{equation}
	
	The proof of \eqref{eq1-13-05-2021} relies on Proposition \ref{prop2-12-05-2021} bound \eqref{eq17-12-05-2021}. In fact we only need to bound $\Box u$. Recall the Sobolev bounds \eqref{eq3b-13-05-2021},
	\begin{equation}\label{eq9-14-05-2021}
		\aligned
		\||\Box u|_{N-1}\|_{L^2(\Hcal_s)}\leq& C\sum_{p_1+p_2=N}\||\del u|_{p_1}|\del v|_{p_2}\|_{L^2(\Hcal_s)} 
		\\
		\leq& CC_1\vep s^{-3/2+\delta} \|(s/t)^{3/2}|\del u|_N\|_{L^2(\Hcal_s)} + CC_1\vep s^{-2+\delta}\||\del v|_{N-1}\|_{L^2(\Hcal_s)}
		\\
		\leq& C(C_1\vep)^2 s^{-3/2+2\delta}.
		\endaligned
	\end{equation}
	Now substitute this bound into \eqref{eq17-12-05-2021} and remark that $C_1\geq C_0$, \eqref{eq2-13-05-2021} is established.
	
	\section{Sharp decay bounds}
	This section is devoted to the following sharp decay bounds.
	\begin{equation}\label{eq12-12-05-2021}
		\aligned
		&|u|_{N-4}\leq CC_1\vep (s/t)s^{-1+CC_1\vep},\quad &&|u|\leq CC_1\vep t^{-1},
		\\
		& |v|_{N-4}\leq CC_1\vep (s/t)^{5/2-2\delta}s^{-3/2+CC_1\vep},\quad &&|v|_{N-4,0}\leq CC_1\vep (s/t)^{5/2-2\delta}s^{-3/2},
		\\
		& |\del v|_{N-4} \leq CC_1\vep (s/t)^{3/2-2\delta}s^{-3/2+CC_1\vep},\quad &&|\del v|_{N-4,0}\leq CC_1\vep (s/t)^{3/2-2\delta}s^{-3/2}.
		\endaligned
	\end{equation}
	\subsection{Sharp bounds on Klein-Gordon component}
	\paragraph*{Direct bounds.}
	In this subsection we will apply Proposition \ref{prop2-11-05-2021} on
	\begin{equation}\label{eq4-12-05-2021}
		\Box \del^IL^J v - P^{\alpha\beta}u\del_{\alpha}\del_{\beta}\del^IL^J v + c^2\del^IL^J v = [\del^IL^J,P^{\alpha\beta}u\del_{\alpha}\del_{\beta}]v
	\end{equation}
	and obtain sharp bounds on $|\del v|_{p,k}$ and $|v|_{p,k}$.
	
	Firstly, we remark that, following the notation of Proposition \ref{prop2-11-05-2021},
	$$
	\Hb = (s/t)^{-2}\Pu^{00}u,\quad P=0.
	$$
	Then recall \eqref{eq1-12-05-2021}, $|\Hb|\leq CC_1\vep\leq 1/2$. Furthermore, \eqref{eq11-11-05-2021} leads to $|H|_{N-2}\leq 1/2$. Thus \eqref{eq4-11-05-2021} is guaranteed.
	
	Furthermore, also by \eqref{eq11-11-05-2021},
	\begin{equation}\label{eq2-12-05-2021}
		|H|_{N-2} \leq C|u|_{N-2}\leq CC_1\vep (s/t)s^{-1+\delta}.
	\end{equation}
	Then thanks to the first bound of \eqref{eq2-06-05-2021} and following the notation of Proposition \ref{prop2-11-05-2021} (with $P\equiv 0$ in the notation therein),
	\begin{equation}\label{eq3-12-05-2021}
		s^{3/2}|S_2[H,P,\del^IL^J v]| \leq  C(C_1\vep)^2(s/t)^{3/2}s^{-2+2\delta},\quad |I|+|J|\leq N-4.
	\end{equation}
	
	Then we turn to the last term in right-hand-side of \eqref{eq5-11-05-2021}. Remark that $\Hu^{00} = H^{\alpha\beta}\Psiu_{\alpha}^0\Psiu_{\beta}^0 = H^{00} - 2(x^a/t)H^{a0} + (x^ax^b/t^2)H^{ab}$. Then $\Lcal(\Hu^{00}) = 0$. Furthermore, $\Lcal(s/t) = 0$. This is because $\Lcal(x^a/t) = 0$. Then Sobolev bounds \eqref{eq10-11-05-2021} and \eqref{eq11-11-05-2021} show that
	$$
	\big|\Lcal\big((1+\Hb)^{-1}\big)\big|\leq CC_1\vep(s/t)^{-2}\big((s/t)|\del u| + (t/s)|\dels u|\big)\leq C C_1\vep (s/t)^{-1}s^{-2+\delta}.
	$$
	Then thanks to the fact that $t^{1/2}\lesssim s$ in $\Kcal$,
	\begin{equation}\label{eq9-11-05-2021}
		\big|\Lcal\big((1+\Hb)^{-1}\big)(s^{3/2}\Lcal \del^IL^J v + s^{1/2}\del^IL^Jv)\big|\leq C(C_1\vep)^2(s/t)^{3/2}s^{-2+2\delta},\quad |I|+|J|\leq N-4.
	\end{equation}
	
	\paragraph*{Bounds on commutators.} Then we turn to the most critical term in \eqref{eq4-12-05-2021}. We firstly remark the following decomposition.
	\begin{equation}\label{eq5-12-05-2021}
		|[\del^IL^J,\del_{\alpha}\del_{\beta}]v|\lesssim |\del\del v|_{p-1,k-1}\lesssim |\del v|_{p,k-1}.
	\end{equation}
	When $k=0$, this commutator disappears. To prove \eqref{eq5-12-05-2021}, we recall the following basic commutation relation:
	$$
	[L^J,\del_{\alpha}]u = \sum_{|J'|<|J|,\beta}\Gamma^{J\beta}_{\alpha J'}\del_{\beta}L^{J'}u
	$$
	where $\Gamma^{J\beta}_{\alpha J'}$ are constants determined by $\alpha,J$. This can be showed by an induction on $|J|$. Then apply this relation twice on $[L^J,\del_{\alpha}\del_{\beta}]v$, then \eqref{eq5-12-05-2021} is obtained. Then for $|I|+|J|\leq N-4$ and $|J| = k$, 
	\begin{equation}\label{eq13-12-05-2021}
		\aligned
		\big|[\del^IL^J,P^{\alpha\beta}u\del_{\alpha}\del_{\beta}]v\big|
		\leq& |u||\del\del v|_{N-5,k-1}+ C\sum_{k_1+k_2=k}|L u|_{k_1-1}|\del\del v|_{N-4-k_1,k_2}
		\\
		&+ C\sum_{p_1+p_2=N-4\atop k_1+k_2=k}|\del u|_{p_1-1,k_1}|\del\del v|_{p_2,k_2}.
		\endaligned
	\end{equation}
	The first two terms in right-hand-side disappear when $k=0$. The last term disappears when $k=N-4$. For the simplicity of expression, we introduce:
	$$
	\Abf_k(s) := \sup_{\Hcal_{[2,s]}}\{t|u|_k\},\quad 
	\Bbf_k(s):= \sup_{\Hcal_{[2,s]}}\big\{ (s/t)^{-5/2+2\delta}s^{3/2}\big((s/t)|\del v|_{N-4,k}+|v|_{N-4,k}\big)\big\}.
	$$
	Substitute the Sobolev bounds into the third term in right-hand-side of \eqref{eq13-12-05-2021} and taking the above notation, we arrive at
	\begin{equation}\label{eq6-12-05-2021}
		\aligned
		s^{3/2}\big|[\del^IL^J,P^{\alpha\beta}u\del_{\alpha}\del_{\beta}]v\big|
		\leq& C(s/t)^{7/2-2\delta}s^{-1} \Abf_0\Bbf_{k-1}
		+ C(s/t)^{7/2-2\delta}s^{-1} \sum_{1\leq k_1\leq k}\Abf_{k_1}\Bbf_{k-k_1} 
		\\
		&+ C(C_1\vep)^2 (s/t)^{3/2}s^{-2+2\delta}
		\endaligned
	\end{equation}
	where we emphasize that the the first terms in right-hand-side disappear when $k=0$.

	\paragraph*{Sharp bound on $v$.} 
	Now we apply Proposition \ref{prop2-11-05-2021} on \eqref{eq4-12-05-2021} with \eqref{eq3-12-05-2021}, \eqref{eq9-11-05-2021} and \eqref{eq6-12-05-2021}. We obtain
	$$
	\aligned
	&s^{3/2}\big((s/t)|\Lcal \del^IL^J v(s)| + |\del^IL^J v(s)|\big)
	\\
	\leq& s_0^{3/2}\sup_{\Hinit_2}\big(|\Lcal v(s)| + |v(s)|\big) + C(C_1\vep)^2(s/t)^{3/2}\int_{\lambda_0}^s\lambda^{-2+2\delta}d\lambda
	\\
	&+ C(s/t)^{7/2-2\delta} \int_{\lambda_0}^s \lambda^{-1} \Abf_0(\lambda)\Bbf_{k-1}(\lambda)d\lambda 
	+  C(s/t)^{7/2-2\delta}\sum_{1\leq k_1\leq k}\int_{\lambda_0}^s \lambda^{-1} \Abf_{k_1}(\lambda)\Bbf_{k-k_1}(\lambda)d\lambda
	\\
	\leq & C(s/t)^{7/2-2\delta} \Big(\int_{\lambda_0}^s \lambda^{-1} \Abf_0(\lambda)\Bbf_{k-1}(\lambda)d\lambda 
	+  \sum_{1\leq k_1\leq k}\int_{\lambda_0}^s \lambda^{-1} \Abf_{k_1}(\lambda)\Bbf_{k-k_1}(\lambda)d\lambda\Big)
	\\
	&+ C(C_1\vep)(s/t)^{3/2}\lambda_0^{-1+2\delta} 
	+
	\left\{
	\aligned
	&CC_0\vep,\quad &&0\leq |x|/t\leq 3/5,
	\\
	& 0,\quad && 3/5\leq |x|/t<1.
	\endaligned
	\right.
	\endaligned
	$$
	Now remark that when $|x|/t\leq 3/5$, one has $4/5\leq s/t\leq 1$. Furthermore, when $|x|/t\geq 3/5$, $\lambda_0 =\sqrt{\frac{t+r}{t-r}}\simeq (s/t)^{-1}$.  Then recall that $C_0\leq C_1$,
	$$
	\aligned
	s^{3/2}\big((s/t)&|\Lcal \del^IL^J v(s)| + |\del^IL^J v(s)|\big)
	\leq CC_1\vep (s/t)^{5/2-2\delta} 
	\\
	&+ C(s/t)^{7/2-2\delta} \Big(\int_{\lambda_0}^s \lambda^{-1} \Abf_0(\lambda)\Bbf_{k-1}(\lambda)d\lambda 
	+\!\!\!\!\!\!  \sum_{1\leq k_1\leq k}\int_{\lambda_0}^s \lambda^{-1} \Abf_{k_1}(\lambda)\Bbf_{k-k_1}(\lambda)d\lambda\Big)
	\endaligned
	$$
	where we emphasize that the last term disappears when $k=0$. Now recall \eqref{eq8-11-05-2021} and obtain:
	\begin{equation}\label{eq7-12-05-2021}
		\Bbf_k(s)\leq CC_1\vep + C(s/t) \Big(\int_{\lambda_0}^s \lambda^{-1} \Abf_0(\lambda)\Bbf_{k-1}(\lambda)d\lambda 
		+\!\!\!\!\!\!  \sum_{1\leq k_1\leq k}\int_{\lambda_0}^s \lambda^{-1} \Abf_{k_1}(\lambda)\Bbf_{k-k_1}(\lambda)d\lambda\Big).
	\end{equation}
	When $k=0$, we obtain
	\begin{equation}\label{eq10-12-05-2021}
		\Bbf_0(s)\leq CC_1\vep \Leftrightarrow (s/t)|\del v| + |v| \leq CC_1\vep(s/t)^{5/2-2\delta}s^{-3/2}.
	\end{equation}
	\subsection{Sharp bounds on wave component and conclusion}
	\paragraph*{Bounds on $u$.}
	We firstly establish the following bound
	\begin{equation}\label{eq9-12-05-2021}
		\Abf_k(s)\leq CC_1\vep + CC_1\eps\Bbf_k(s),
	\end{equation}
	where $\Abf_k, \Bbf_k$ are defined as in the last subsection. This is by applying Proposition \ref{prop1-12-05-2021} on
	$$
	\Box L^J u = L^J\big(B^{\alpha\beta}\del_{\alpha}u\del_{\beta}v\big). 
	$$
	For this purpose, we decompose $L^J u$ as following:
	$$
	L^J u = w^J_{\init} + w^J_{\source}, \quad\text{with}\quad \Box w^J_{\init} = 0,\quad \del_t w^J_{\init}(2,x) = \del_tL^Ju(2,x),\,\, w^J_{\init}(2,x) = L^Ju(2,x).
	$$
	Then
	\begin{equation}\label{eq8-13-05-2021}
		\Box w^J_{\source} =  L^J\big(B^{\alpha\beta}\del_{\alpha}u\del_{\beta}v\big),
		\quad
		w^J_{\source}(2,x) = \del_t w^J_{\source}(2,x) = 0.
	\end{equation}
	It is clear that
	$$
	\aligned
	|\Box w^J_{\source}|\leq CC_1\eps s^{-2+\delta} \, (s/t)^{3/2+2\delta} s^{-3/2}\Bbf_k(s)
	\leq  CC_1\eps t^{-5/2-\delta/2}(t-r)^{-1+3\delta/2}\Bbf_k(s).
	\endaligned
	$$
	Now we apply Proposition \ref{prop1-12-05-2021} in $\Hcal_{[2,s]}$ and obtain
	$$
	\sup_{\Hcal_{[2,s]}}\{t |w^J_{\source}|\}\leq CC_1\eps \Bbf_k(s).
	$$
	Also remark that $\del_t w^J_{\init}(2,x),  w^J_{\init}(2,x)$ are compactly supported $C^2$ functions. Then $w^J_{\init}\leq CC_0\vep t^{-1}$. Then \eqref{eq9-12-05-2021} is established.
	\paragraph*{Induction and conclusion.}
	Now we write \eqref{eq7-12-05-2021} and \eqref{eq9-12-05-2021} together. This forms a system of inequalities. We firstly remark that by fixing $k=0$ in \eqref{eq9-12-05-2021} and recalling \eqref{eq10-12-05-2021},
	\begin{equation}\label{eq11-12-05-2021}
		\Abf_0(s) + \Bbf_0(s) \leq CC_1\vep.
	\end{equation}
	Then by induction and Gronwall's inequality, we conclude that
	\begin{equation}
		\Abf_{N-4}(s) + \Bbf_{N-4}(s) \leq CC_1\vep s^{CC_1\vep}.
	\end{equation}
	This concludes \eqref{eq12-12-05-2021}.
	
	\section{Improvement of energy bounds and conclusion}
	\subsection{Bounds on source terms}
	We are going to apply Theorem \ref{thm1-29-04-2021} and Proposition \ref{prop2-12-05-2021}. For the wave component we differentiate the wave equation of \eqref{eq4-13-05-2021} with respect to $\del^IL^J$ with $|I|+|J|\leq N, |J|\leq k$:
	\begin{equation}\label{eq8-14-05-2021}
		\Box\del^IL^J u =  \del^IL^J(B^{\alpha\beta}\del_{\alpha}u\del_{\beta}v).
	\end{equation}
	Then we need to bound $\|(s/t)^{1/2}\del^IL^J(B^{\alpha\beta}\del_{\alpha}u\del_{\beta}v)\|_{L^2(\Hcal_s)}$. In fact based on \eqref{eq3b-13-05-2021} and Lemma~\ref{lem1-14-05-2021}, we can prove that
	\begin{equation}\label{eq15-12-05-2021}
		s^{1/2}\|(s/t)^{1/2}|\Box u|_{N,k}\|_{L^2(\Hcal_s)}
		\leq  CC_1\eps s^{-1}\big(\Ebf_1^{N,k}(s,u)^{1/2} + \Ebf_{0,c}^{N,k}(s,v)^{1/2}\big).
	\end{equation}
	This is because, recalling \eqref{eq8-14-05-2021}, 
	$$
	|\Box u|_{N,k}\leq C\sum_{p_1+p_2=N\atop k_1+k_2=k}|\del u|_{p_1,k_1}|\del v|_{p_2,k_2}.
	$$
	When $p_1\leq N-3$, we apply \eqref{eq14-12-05-2021} and obtain
	$$
	\|(s/t)^{1/2}|\del u|_{p_1,k_1}|\del v|_{p_2,k_2}\|_{L^2(\Hcal_s)}\leq CC_1\eps s^{-3/2} \|(s/t)|\del v|_{N,k}\|_{L^2(\Hcal_s)} = CC_1\eps s^{-3/2} \Ebf_{0,c}^{N,k}(s,v)^{1/2},
	$$
	and, when $p_1\geq N-2$ which implies $p_2\leq 2\leq N-6$, we apply \eqref{eq12-12-05-2021} written in the following form
	\begin{equation}
	|\del\del v|_{N-6,k}\leq CC_1\vep (s/t)^{5/2}s^{-1+CC_1\vep},\quad |\del\del v|_{N-6,0}\leq CC_1\vep s^{-3/2}. 
	\end{equation}
	together with Lemma~\ref{lem1-14-05-2021}:
	$$
\|(s/t)^{1/2}|\del u|_{p_1,k_1}|\del v|_{p_2,k_2}\|_{L^2(\Hcal_s)}
\leq CC_1\eps s^{-2+CC_1\eps}\|(s/t)^3 t^{1/2}|\del u|_{N,k}\|_{L^2(\Hcal_s)} 
\leq C(C_1\eps) s^{-3/2}\Ebf_1^{N,k}(s,u)^{1/2}.
	$$
	
	On the other hand, for Klein-Gordon equation, we recall \eqref{eq13-12-05-2021}. For the last term in right-hand-side we only need the Sobolev  decay and it is bounded as 
	$$
	\||\del u|_{p_1-1,k_1}|\del\del v|_{p_2,k_2}\|_{L^2(\Hcal_s)}\leq C(C_1\vep)s^{-1}\Ebf_{0,c}^{N,k}(s,v)^{1/2} + C(C_1\eps)^2 s^{-2+2\delta}.
	$$
	For the rest terms, we need the sharp decay \eqref{eq12-12-05-2021}. For the second term in right-hand-side of \eqref{eq13-12-05-2021}, remark that $k_2+k_1=k$ and $k_1\geq 1$, then $k_2\leq k-1$. Provided that $N\geq 9$,
	$$
	\aligned
	&\||Lu|_{k_1-1}|\del\del v|_{p_2,k_2}\|_{L^2(\Hcal_s)}
	\\
	\leq& \|(s/t)s^{-1+CC_1\vep}|\del\del v|_{N,k-1}\|_{L^2(\Hcal_s)} 
	\\
	&+ \|(s/t)^{5/2-2\delta}s^{-3/2}|u|_{k}\|_{L^2(\Hcal_s)} + \|(s/t)^{5/2-2\delta}s^{-3/2+C_1\vep}|u|_{k-1}\|_{L^2(\Hcal_s)}
	\\
	\leq&  CC_1\vep s^{-1+CC_1\vep}\Ebf_{0,c}^{N,k-1}(s,v)^{1/2}
	+ CC_1\vep s^{-1}\|(s/t)^{1-2\delta}\,(s/t)t^{-1/2}|u|_{k}\|_{L^2(\Hcal_s)} 
	\\
	&+ CC_1\vep s^{-1+C_1\vep}\|(s/t)^{1-2\delta}\,(s/t)t^{-1/2}|u|_{k-1,k-1}\|_{L^2(\Hcal_s)}
	\\
	\leq& CC_1\vep s^{-1}\Ebf_1^{N,k}(s,u)^{1/2} +  CC_1\vep s^{-1+CC_1\vep}\big(\Ebf_{0,c}^{N,k-1}(s,v)^{1/2} + \Ebf_1^{N,k-1}(s,u)^{1/2}\big),
	\endaligned
	$$
	where for the last inequality we have applied the estimate
	$$
	\|(s/t)t^{-1/2}|u|_{p,k}\|_{L^2(\Hcal_s)}\leq C\Ebf_1^{p,k}(s,u)^{1/2}
	$$
	guaranteed by Lemma~\ref{lem1-14-05-2021}. The first term in right-hand-side of \eqref{eq13-12-05-2021} is bounded in the same manner, we omit the detail. Then we conclude that
	\begin{equation}\label{eq16-12-05-2021}
		\aligned
		\|[\del^IL^J,P^{\alpha\beta}u\del_{\alpha}\del_{\beta}]v\|_{L^2(\Hcal_s)}
		\leq&  CC_1\vep s^{-1}\big(\Ebf_1^{N,k}(s,u)^{1/2} + \Ebf_{0,c}^{N,k}(s,v)^{1/2}\big)
		\\
		&+  CC_1\vep s^{-1+CC_1\vep}\big(\Ebf_1^{N,k-1}(s,u)^{1/2} + \Ebf_{0,c}^{N,k-1}(s,v)^{1/2}\big).
		\endaligned
	\end{equation}
	We emphasize that the last term in RHS does not exist when $k=0$. 
	\subsection{Energy estimates}
	For wave equation, we apply directly Theorem \ref{thm1-29-04-2021} together with \eqref{eq15-12-05-2021} and obtain:
	\begin{equation}\label{eq21-12-05-2021}
		\Ebf_1^{N,k}(s, u)^{1/2}\leq \Ebf_1^{N,k}(s_0,u)^{1/2} + CC_1\vep \int_{s_0}^s {s^{\prime}}^{-1}\big(\Ebf_1^{N,k}(s^{\prime},u)^{1/2} + \Ebf_{0,c}^{N,k}(s^{\prime},v)^{1/2}\big)ds^{\prime}.
	\end{equation}
	
	For the Klein-Gordon component, we need to guarantee \eqref{eq22-12-05-2021}. We firstly recall \eqref{eq23-12-05-2021} and the Sobolev decay \eqref{eq1-12-05-2021} implies
	$$
	|\Ebf_{0,g,c}(s,v) - \Ebf_{0,c}(s,v)| \leq C\int_{\Hcal_s}|H||\del v|^2 \, dx\leq CC_1\vep\int_{\Hcal_s}(s/t)^2|\del v|^2\,dx\leq CC_1\vep \Ebf_{0,c}(s,v)^{1/2}.
	$$
	Taking $CC_1\vep\leq 3/4$, we obtain \eqref{eq18-12-05-2021} with $\kappa = 2$. Furthermore, with the notation of Proposition \ref{prop2-12-05-2021},
	$$
	\bigg|\int_{\Hcal_s}(s/t)\Big(\del_\alpha H^{\alpha\beta}
	\del_t v\del_\beta v - \frac{1}{2}\del_t H^{\alpha\beta}\del_\alpha v\del_\beta v\Big) \, dx\bigg|
	\leq \int_{\Hcal_s}(s/t)|H||\del v|^2\leq CC_1\vep s^{-1}\Ebf_{0,c}(s,v),
	$$
	that is, 
	\begin{equation}\label{eq1-15-05-2021}
		M(s) = CC_1\vep s^{-1} \Ebf_{0,c}(s,v)^{1/2}.
	\end{equation}
	Then we apply Proposition \ref{prop2-12-05-2021} and obtain
\begin{equation}\label{eq1-29-dec-2023}
	\aligned
	\Ebf_{0,c}^{N,k}(s,v)^{1/2}\leq& 4\Ebf_{0,c}^{N,k}(s_0,v)^{1/2} 
	+ CC_1\vep\int_2^s {s^{\prime}}^{-1}\big(\Ebf_1^{N,k}(s,u)^{1/2} +\Ebf_{0,c}^{N,k}(s^{\prime},v)^{1/2}\big)ds^{\prime}
	\\
	&+CC_1\vep\int_2^s {s^{\prime}}^{-1+CC_1\vep}\big(\Ebf_1^{N,k-1}(s,u)^{1/2} +\Ebf_{0,c}^{N,k-1}(s^{\prime},v)^{1/2}\big)ds^{\prime}.
	\endaligned
\end{equation}
	Taking the sum of \eqref{eq1-29-dec-2023} and \eqref{eq21-12-05-2021}, we obtain (taking $s_0 = 2$ and apply \eqref{eq25-12-05-2021})
	\begin{equation}\label{eq24-12-05-2021}
		\Ebf_k(s) \leq C_0\vep + CC_1\vep\int_2^s{s^{\prime}}^{-1}\Ebf_k(s^{\prime})ds^{\prime} + CC_1\vep \int_2^s{s^{\prime}}^{-1+C\sqrt{C_1\vep}}\Ebf_{k-1}(s^{\prime})ds^{\prime}
	\end{equation}
	with $\Ebf_k(s): =\Ebf_1^{N,k}(s,u)^{1/2} +\Ebf_{0,c}^{N,k}(s,v)^{1/2}$. The last term in right-hand-side does not exist when $k=0$. Then by Gronwall's inequalitywe obtain
	$$
	\Ebf_0(s)\leq C_0 \eps + CC_1\eps s^{CC_1\eps}\leq  CC_1\eps s^{C\sqrt{C_1\eps}}.
	$$
	Substitute this again into \eqref{eq24-12-05-2021} for the case $k=0$, we obtain
	$$
	\Ebf_1^{N,0}(s,u)^{1/2} +\Ebf_{0,c}^{N,0}(s,v)^{1/2} \leq C_0\vep + C(C_1\vep)^{3/2} s^{C\sqrt{C_1\vep}}.
	$$
	Then by induction and Gronwall's inequality, 
	$$
	\Ebf_k(s)\leq C_0\eps + C(C_1\eps)^{3/2} s^{C\sqrt{C_1\eps}}.
	$$
	This implies 
	\begin{equation}\label{eq26-12-05-2021}
		\Ebf_1^{N}(s,u)^{1/2} + 4\Ebf_{0,c}^{N}(s,v)^{1/2} \leq 5C_0\vep + C(C_1\vep)^{3/2} s^{C\sqrt{C_1\vep}}.
	\end{equation}
	Now we fix $C_1>10C_0$ and 
	\begin{equation}\label{eq6-19-08-2022-M}
		\vep\leq \delta^2/CC_1,\quad  \vep \leq \bigg(\frac{C_1-10C_0}{2CC_1^{3/2}}\bigg)^2,
	\end{equation}
	then
	\begin{equation}\label{eq2-15-05-2021}
		\Ebf_1^{N}(s,u)^{1/2} + 4\Ebf_{0,c}^{N}(s,v)^{1/2} \leq\frac{1}{2}C_1\vep s^{C\sqrt{C_1\vep}}\leq \frac{1}{2}C_1\vep s^{\delta},
	\end{equation}
	which closes the bootstrap argument which guarantees the global existence.
	
	\section{Energy bounds on time-constant hyperplane}
	The bootstrap argument only gives the energy bounds on hyperboloids. In this section we show that the associate energies defined on time-constant hyperplane can be bonded by the hyperboloidal energies. This is done in \cite{PLF-YM-book1} when we construct the initial data. Here we give a brief review.
	
	For the convenience of discussion, we introduce
	$$
	\Dcal^+_{t_0}:=\Kcal\cap \big\{t_0\leq t\leq (t_0^2 + r^2)^{1/2}\big\},\quad \Dcal^-_{t_0} = \Kcal\cap \{(s_{t_0}^2+r^2)^{1/2}\leq t\leq t_0\}.
	$$
	with $s_{t_0} = \sqrt{2t_0-1}$. We remark that $\Dcal^+_{t_0}$ is the sub-region of $\Kcal$ bounded by one time-constant hyperplane and the hyperboloid on it, $\Dcal^-_{t_0}$ is the the sub-region of $\Kcal$ bounded by one hyperboloid $\Hcal_{s_{t_0}}$ and a time-constant hyperplane determined by the circle $\del\Kcal\cap \Hcal_{s_{t_0}}$.
	We also introduce the standard energy defined on hyperplane:
	$$
	\Ebf^{\Pcal}_{0,g,c}(t,u):=\int_{\RR^3}\big(g^{00}|\del_t u|^2 - g^{ab}\del_au\del_bu + c^2u^2\big)(t,\cdot) \,dx.
	$$
	When the metric is flat, we simplify the notation as $\Ebf_{0,c}^{\Pcal}(t,u)$ and when $c=0$, we write $\Ebf_0^{\Pcal}(t,u)$. 
	
	Consider the standard energy multiplier $\del_t u$ acting on $g^{\alpha\beta}\del_{\alpha}\del_{\beta} u + c^2u$, we obtain
	$$
	\aligned
	\del_t u(g^{\alpha\beta}\del_{\alpha}\del_{\beta}u +c^2u) 
	= &\frac{1}{2}\del_t\big(g^{00}|\del_t u|^2 - g^{ab}\del_au\del_bu + c^2u^2\big)
	+\del_a\big(g^{a\beta}\del_tu\del_{\beta}u\big) 
	\\
	&-\del_{\alpha}g^{\alpha\beta}\del_tu\del_{\beta}u +\frac{1}{2}\del_tg^{\alpha\beta}\del_{\alpha}u\del_{\beta}u.
	\endaligned
	$$ 
	Integrate the above divergence form in $D^+_{t_0}$,
	$$
	\Ebf_{0,g,c}(t,u) - \Ebf^{\Pcal}_{0,g,c}(t,u) = \int_{\Dcal^+_t}\del_tu f \, dxdt
	+ \int_{\Dcal^+_t}\big(\del_{\alpha}g^{\alpha\beta}\del_tu\del_{\beta}u -\frac{1}{2}\del_tg^{\alpha\beta}\del_{\alpha}u\del_{\beta}u\big)\,dxdt
	$$ 
	The right-hand-side is bounded as follows.
	$$
	\Big|\int_{\Dcal^+_t}\del_tu f \, dxdt\Big|\leq \int_{\Dcal^+_t}|\del_tu f|\,dxdt \leq \int_{\Dcal^+_t\cup \Dcal^-_t}|\del_t u f|\,dxdt
	= \int_{\Hcal_{[s_t,t]}}|\del_tuf|\,dxdt.
	$$
	This leads to
	$$
	\Big|\int_{\Dcal^+_t}\del_tu f \, dxdt\Big|\leq C\int_{s_t}^t\Ebf_{0,c}(s,u)^{1/2}\|f\|_{\Hcal_s}ds
	$$
	In the same manner, 
	$$
	\Big|\int_{\Dcal^+_t}\big(\del_{\alpha}g^{\alpha\beta}\del_tu\del_{\beta}u -\frac{1}{2}\del_tg^{\alpha\beta}\del_{\alpha}u\del_{\beta}u\big)\,dxdt\Big|\leq \int_{s_t}^tM(s)\Ebf_{0,c}(s,u)^{1/2}ds,
	$$
	where $M(s)$ is defined in Proposition \ref{prop2-12-05-2021}. This leads to
	\begin{equation}\label{eq1-14-05-2021}
		\big|\Ebf_{0,g,c}(t,u) - \Ebf^{\Pcal}_{0,g,c}(t,u)\big|\leq C\int_{s_t}^t \big(M(s) + \|f\|_{L^2(\Hcal_s)} \big)\Ebf_{0,c}(s,u)^{1/2}ds.
	\end{equation}
	
	Return to the case of system \eqref{eq4-13-05-2021}. For the wave equation,  we consider \eqref{eq8-14-05-2021}. In this case $M(s)\equiv 0$ and $f$ is bounded by \eqref{eq9-14-05-2021} which is integrable. Then
\begin{equation}\label{eq1-26-dec-2023}
	\Ebf^{\Pcal}_0(t,\del^IL^J u)\leq \Ebf_0(t,\del^IL^J u) + C(C_1\vep)^3\leq CC_1\vep.
\end{equation}
	
	
	\section{The Friedlander radiation field and rigidity}
	\subsection{Decomposition of the D'Alembert operator}
	In this section we will firstly show that the Friedlander radiation fined is well defined for the wave component $u$. Furthermore, we will prove that when $\mathcal{R}_{u_t}(\mu,\omega)\equiv 0$, one must has $\vep = 0$, provided that  $\vep$ sufficiently small.
	
	For the first problem, we rely on the decay estimate by integration along time-like hyperbolas (see \cite{M-2021-strong}). Here for the convenience of application we give a modified version.
	
	We first recall the following decomposition established in \cite{M-2021-strong} for wave operator in $\RR^{1+n}$:
	\begin{equation}\label{decompo-wave-2}
		\aligned
		\Box u =& (t-r)^{-\beta}t^{-\alpha}\big((s/t)^2\del_t 
		+ (2x^a/t)\delu_a\big)((t-r)^{\beta}t^{\alpha}\del_t u) + p_{n,\alpha,\beta}(t,r)\del_tu 
		\\
		&- \sum_a \delu_a\delu_a u
		\endaligned
	\end{equation}
	where
	$$
	p_{n,\alpha,\beta}(t,r) = \big((n-\alpha) - (\alpha+1)(r/t)^2 - \beta(t-r)t^{-1}\big)t^{-1}.
	$$
	In order to keep $p_{n,\alpha,\beta}$ positive, one needs:
	$$
	\alpha \leq \frac{n-1}{2},\quad \beta\leq n-\alpha.
	$$
	In the present case we have $n=3$ and we fix $\alpha = 1$, $\beta = 0$. Then \eqref{decompo-wave-2} is written as
	\begin{equation}\label{eq2-19-08-2022-M}
		\mathcal{J}\big(t \del_tu\big) + P\ \big(t \del_tu\big) = S^w[u] + \Delta^w[u]	
	\end{equation}
	with
	$$
	\aligned
	&\mathcal{J} := \del_t + \frac{2tx^a}{t^2+r^2}\del_a,
	\\
	&P(t,r) := \frac{t^2}{t^2+r^2}p(t,r) = t^{-1}\frac{2s^2}{t^2+r^2},
	\\
	&S^w[u] := \frac{t^3\Box u}{t^2+r^2},\quad 
	\Delta^w[u] := \frac{t^3\sum_a\delu_a\delu_a u}{t^2+r^2}.
	\endaligned
	$$
	Let $(t_0,x_0)\in\Kcal$ and $\gamma(\cdot;t_0,x_0)$ be the integral curve of $\Jcal$ with $\gamma(t_0;t_0,x_0) = (t_0,x_0)$. As mentioned in \cite{M-2021-strong}, 
	\begin{equation}\label{eq1-18-08-2020}
		\aligned
		&\gamma(t;t_0,x_0) = \big(\gamma^\alpha(t;t_0,x_0)\big)_{\alpha=0,1,2,3},
		\\
		& \gamma^0(t;t_0,x_0) = t,\quad
		\gamma^a(t;t_0,x_0) = (x_0^a/r_0)\left(\sqrt{t^2+\frac{1}{4}c_0^2} - \frac{1}{2}c_0\right).
		\endaligned
	\end{equation}
	Or equivalently, along $\gamma(\cdot;t_0,x_0)$, 
	\begin{equation}\label{eq13-19-08-2022-M}
		\frac{t^2-r^2}{r} = c_0.
	\end{equation}
	These are time-like hyperbolas. Before we go further, we remark the following basic results which can be checked directly.
	\begin{lemma}\label{lem1-19-08-2022-M}
		Let $\gamma(\cdot;t_0,x_0)$ be an integral curve of $\Jcal = \del_t + \frac{2tx^a}{t^2+r^2}\del_a$ with $(t_0,x_0)\in\Kcal$ with $r_0\neq 0$. Then
		\\
		$\bullet$ the quantity $(r/t)$ is strictly increasing along $\gamma(\cdot;t_0,x_0)$ with respect to $t$. 
		\\
		$\bullet$ let $c_0 = \frac{t_0^2-r_0^2}{r_0}$. Then
		$$
		\lim_{t\rightarrow \infty} (t-r)|_{\gamma(\tau;t_0,x_0)} = c_0/2,
		$$
		i.e., $\big(\tau, (\tau-c_0/2)(x_0^a/r_0)\big)$ is one of the asymptote of $\gamma(\tau;t_0,x_0)$. Furthermore,
		\begin{equation}\label{eq3-19-08-2022-M}
			\lim_{\tau\rightarrow \infty}\tau\big(r|_{\gamma(\tau;t_0,x_0)} - (\tau-c_0/2) \big) = \frac{1}{8}c_0^2.
		\end{equation}
		\\
		$\bullet$ For each $\gamma(\cdot;t,x)$:
		\\
		- when $c_0\leq 8/3$, there exists a unique $(t_0,x_0)\in \del\Kcal$ such that $\gamma(t_0;t,x) = (t_0,x_0)$; \\
		- when $c_0\geq 8/3$, there exists a unique $(t_0,x_0)\in\Hcal_2\cap\{r\leq t-1\}$ such that $\gamma(t_0;t,x) = (t_0,x_0)$.
	\end{lemma}
	
	\subsection{Existence of the Friedlander radiation field}
	Let us denote by $U_{t,x}(\tau) := t\del_tu\big|_{\gamma(\tau;t,x)}$. Then \eqref{eq2-19-08-2022-M} leads to
	\begin{equation}\label{eq12-19-08-2022-M}
		U_{t,x}'(\tau) + P_{t,x}(\tau) U_{t,x}(\tau) = \big(S^w[u] + \Delta^w[u]\big)\big|_{\gamma(\tau;t,x)}.
	\end{equation}
	Here $P_{t,x}(\tau) := P|_{\gamma(\tau;t,x)}$.
	
	Then by integrating the above ODE, we have, for $(t,x)\in \Hcal_2$,
	\begin{equation}\label{eq21-19-08-2022-M}
		U_{t,x}(T) = U_{t,x}(t)e^{-\int_{t}^TP_{t,x}(\eta) d\eta} 
		+ \int_t^T\big(S^w[u] + \Delta^w[u]\big)\big|_{\gamma(\tau;t,x)} e^{-\int_{\tau}^TP_{t,x}(\eta) d\eta}d\tau. 
	\end{equation}
	
	Then we establish the following result.
	\begin{lemma}\label{lem1-20-08-2022-M}
		Let $(u,v)$ be the global solution to \eqref{eq4-13-05-2021}. Then $\lim_{T\rightarrow\infty}U_{t,x}(T)$ exists for each $(t,x)\in\Kcal$. Furthermore, 
		\begin{equation}
			\lim_{T\to +\infty} |U_{t,x}(T)|\leq CC_1\vep 
		\end{equation}
		where $C_1\vep$ is the global energy bounds of the global solution established by the bootstrap argument.
	\end{lemma}
	\begin{proof}
		We rely on \eqref{eq14-12-05-2021}, \eqref{eq12-12-05-2021} and \eqref{eq21-19-08-2022-M}. Remark that
		$$
		\aligned
		&|\delu_a\delu_a u|\leq Ct^{-2}|u|_2\leq CC_1\vep t^{-3+\delta},
		\\
		&|\del u\del v|\leq C(C_1\eps)^2 s^{-2+\delta} (s/t)^{5/2-2\delta}s^{-3/2}\leq C(C_1\eps)^2 (s/t)^{5/2-2\delta}s^{-7/2+\delta}.
		\endaligned
		$$
	Then we obtain 
		\begin{equation}\label{eq2-20-08-2022-M}
			|S^w[u] + \Delta^w[u]|\leq CC_1\vep t^{-2+\delta}.
		\end{equation}
		Thus for $T_2\geq T_1\geq T$,
		\begin{equation}
			\aligned
			|U_{t,x}(T_2) - U_{t,x}(T_1)| 
			\leq&e^{-\int_t^{T_1}P_{t,x}(\eta)d\eta}|U_{t,x}(t)|\big(1-e^{-\int_{T_1}^{T_2}P_{t,x}(\eta)d\eta}\big)
			\\
			&+\int_t^{T_1}\big|S^w[u] + \Delta^w[u]\big|_{\gamma(\tau;t,x)} \big(1-e^{-\int_{T_1}^{T_2}P_{t,x}(\eta)d\eta}\big)d\tau
			\\
			&+\int_{T_1}^{T_2}\big|S^w[u] + \Delta^w[u]\big|_{\gamma(\tau;t,x)}d\tau.
			\endaligned	
		\end{equation}
		The last term is bounded by $CC_1\vep T^{-1+\delta}$. For the first two terms, we remark that
		$$
		0\leq P(t,r) = \frac{2}{t}\frac{r}{t^2+r^2}\frac{t^2-r^2}{r} = \frac{2(r/t)}{t^2+r^2}\frac{s^2}{r}\leq Ct^{-2}(s^2/r).
		$$
		Recalling \eqref{eq13-19-08-2022-M}, one has, for fixed $c_0 = \frac{t^2-r^2}{r}$,
		$$
		0\leq 1-e^{-\int_T^{\infty}P_{t,x}(\eta)d\eta}\leq 1-e^{Cc_0T^{-1}}.
		$$
		Thus for $T_2\geq T_1\geq T$,
		\begin{equation}\label{eq1-20-19-2022-M}
			0\leq 1-e^{-\int_{T_1}^{T_2}P_{t,x}(\eta)d\eta}\leq 1-e^{Cc_0 T^{-1}}.
		\end{equation}
		This bound together with \eqref{eq2-20-08-2022-M} leads us to the fact that $|U_{t,x}(T_2) - U_{t,x}(T_1)|\rightarrow 0$ provided that $T\rightarrow \infty$. Thus the limit $\lim_{T\rightarrow \infty}U_{t,x}(T)$ exists. 
		
		The upper bound is established in the same manner. We only need to substitute \eqref{eq2-20-08-2022-M} into \eqref{eq21-19-08-2022-M}, and integrate the inequality from $(t_0,x_0)$.
	\end{proof}
	
	Now we are ready to establish the first main result of this section.
	\begin{proposition}\label{prop1-02-06-2023}
		Let $(u,v)$ be the solution to \eqref{eq4-13-05-2021} with \eqref{eq19-19-08-2022-M} holds for sufficient small $\vep$.  The Friedlander radiation field is well-defined. Furthermore,
		\begin{equation}\label{eq3-20-08-2022-M}
			\lim_{T\rightarrow \infty}U_{t,x}(T) = \mathcal{R}_{u_t}(c_0/2,x/r)
		\end{equation}
		with $c_0 = \frac{t^2-r^2}{r}$.
	\end{proposition}
	\begin{proof}
		We make the following calculation:
		$$
		\aligned
		&(t\del_t u)|_{\gamma(\tau;t,x)} - (r\del_tu)|_{(\tau,(\tau-c_0/2)(x/r))}
		\\
		=& \tau \big(\del_tu|_{\gamma(\tau;t,x)} - \del_tu_{(\tau,(\tau-c_0/2)(x/r))}\big) + (c_0/2)\del_tu|_{(\tau,(\tau-c_0/2)(x/r))}.
		\endaligned
		$$
		Now recall \eqref{eq3-19-08-2022-M}, the first term in the right-hand side of the above identity is bounded by
		$$
		C\sup_{\{t=\tau\}\cap\Kcal}\{|\del\del u|\}
		$$
		which converges to zero when $\tau\rightarrow \infty$. The last term is bounded by 
		$$
		C\sup_{\{t=\tau\}\cap \Kcal}\{|\del u|\},
		$$ 
		thus also converges to zero. Then
		$$
		\lim_{T\rightarrow\infty} \Big(U_{t,x}(T) - (r\del_tu)\big|_{(\tau, (\tau-c_0/2)(x/r))}\Big) = 0.
		$$
		Then recall Lemma~\ref{lem1-20-08-2022-M}, we conclude by the desired result.
	\end{proof}

	\subsection{Excessive decay estimate}
We firstly establish the following result.
\begin{lemma}\label{lem1-29-dec-2023}
	Let $(u,v)$ be the global solution to \eqref{eq4-13-05-2021}-\eqref{eq19-19-08-2022-M} with smallness conditions \eqref{eq1-28-dec-2023}. Suppose that there exists $\eta<1$ such that in $\{r\geq \eta t \}\cap\Kcal$,
	\begin{equation}\label{eq4-19-08-2022-M}
	|\del_t u|\leq CC_1\vep t^{-1/2-\delta}s^{-1}
	\end{equation}
	with $\delta>0$. Then 
	\begin{equation}
	\lim_{s\rightarrow +\infty}\big(s^{2\sigma}\Ebf_0(s,u)\big) = 0,\quad \forall\, 0\leq \sigma<\delta.
	\end{equation}
\end{lemma}
\begin{proof}
	We remark that by \eqref{eq11-11-05-2021},
	\begin{equation}\label{eq5-19-08-2022-M}
	|\delb_a u|^2\leq C(C_1\eps)^2(s/t)^4s^{-4+2\delta}\leq C(C_1\vep)^2 s^{-4+2\delta}\big(1+(r/s)^2\big)^{-2+\delta},
	\end{equation}
	where \eqref{eq12-12-05-2021} (the first one) and \eqref{eq6-19-08-2022-M} are applied. Thus
	$$
	\aligned
	\int_{\Hcal_s}|\delu_a u|^2dx \leq& C(C_1\vep)^2s^{-4+2\delta}\int_0^{\frac{s^2-1}{2}}(1+(r/s)^2)^{-2+\delta}r^2\,dr
	\\
	\leq& C(C_1\vep)^2s^{-1+2\delta}\int_0^{\infty}(1+(r/s)^2)^{-2+\delta}d(r/s)
	\\
	\leq& C(C_1\vep)^2s^{-1+2\delta}.
	\endaligned
	$$
	Thus we remark that
	$$
	\aligned
	\Ebf_0(s,u) =& \int_{\Hcal_s}(s/t)^2|\del_t u|^2 + \sum_a|\delu_a u|^2 dx 
	\\
	=& \int_{\Hcal_s\cap\{r\leq \eta t\}} 
	+ \int_{\Hcal_s\cap\{r\geq \eta t\}} (s/t)^2|\del_t u|^2dx + \sum_a\int_{\Hcal_s}|\delu_a u|^2 dx
	\\
	=:& \Ebf_0^{\text{int}}(s,u) + \Ebf_0^{\text{ext}}(s,u) + \sum_a\int_{\Hcal_s}|\delu_a u|^2 dx
	\\
	\leq& \Ebf_0^{\text{int}}(s,u) + \Ebf_0^{\text{ext}}(s,u) + C(C_1\eps)^2s^{-1+\delta}.
	\endaligned
	$$
	For $\Ebf_0^{\text{int}}(s,u)$, remark that $(s/t)^2\geq 1-\eta^2>0$. Thus by \eqref{eq10-11-05-2021},
	$$
	(s/t)^2|\del_t u|^2\leq C(C_1\vep)^2(1-\eta^2)^{-1} s^{2\delta}t^{-4}\leq C(C_1\vep)^2(1-\eta^2)^{-1} s^{-4+2\delta}\big(1+(r/s)^2\big)^{-2}.
	$$
	Then similar to the term $|\delu_a u|^2$, one obtains
	$$
	\Ebf_0^{\text{int}}(s,u)\leq C(C_1\vep)^2(1-\eta^2)^{-1}s^{-1+2\delta}.
	$$
	
	Now for $\Ebf_0^{\text{ext}}(s,u)$, we recall \eqref{eq4-19-08-2022-M},
	$$
	(s/t)^2|\del_t u|^2\leq C(C_1\vep)^2s^{-3-2\delta}\big(1 + (r/s)^2\big)^{-3/2-\delta}.
	$$
	Then
	$$
	\aligned
	\Ebf_0^{\text{ext}}(s,u)\leq& C(C_1\vep)^2s^{-3-2\delta}\int_{\Hcal_s\cap\{r\geq \eta t\}}\frac{r^2dr}{(1+(r/s)^2)^{3/2+\delta}}
	\\
	\leq& C(C_1\vep)^2s^{-2\delta}\int_0^{\infty}(1+(r/s)^2)^{-1/2-\delta}d(r/s)
	\\
	\leq& C(C_1\vep)^2s^{-2\delta}.
	\endaligned
	$$
	So we conclude that
	\begin{equation}\label{eq7-19-08-2022-M}
	\lim_{s\rightarrow \infty}\Ebf_0(s,u) = 0.
	\end{equation}
\end{proof}

Then we establish the decay rate\eqref{eq4-19-08-2022-M}.
\begin{lemma}\label{lem2-29-dec-2023}
	Let $(u,v)$ be the global solution to \eqref{eq4-13-05-2021}-\eqref{eq19-19-08-2022-M} with smallness conditions \eqref{eq1-28-dec-2023}. Suppose further that the Fiedlander radiation field associate to $u$ vanishes at the null infinity. Then 
	\begin{equation}\label{eq18-19-08-2022-M}
	|\del_tu(t,x)|\leq CC_1\vep t^{-2+\delta}.
	\end{equation}
\end{lemma}	
\begin{proof}	
	We mark that 
	$$
	P(t,r) = \frac{2}{t}\frac{r}{t^2+r^2}\frac{t^2-r^2}{r}.
	$$
	When evaluating along $\gamma(\cdot;t,x)$, the last factor is constant (due to \eqref{eq13-19-08-2022-M}). Then
	\begin{equation}\label{eq14-19-08-2022-M}
	0\leq P_{t,x}(\tau)\leq C(s^2/r) (1+\tau)^{-2},\quad s^2 = t^2-r^2.
	\end{equation}
	
	Now we write \eqref{eq21-19-08-2022-M} into the following form
	\begin{equation}\label{eq16-19-08-2022-M}
	U_{t,x}(T)e^{\int_t^TP(t,x)(\eta)d\eta} = U(t,x)(t)
	+ \int_t^T \big(S^w[u] + \Delta^w[u]\big)\big|_{\gamma(\tau;t,x)}e^{\int_t^{\tau}P_{t,x}(\eta)d\eta}d\tau
	\end{equation}
	For the left-hand side, we recall \eqref{eq14-19-08-2022-M} and obtain
	$$
	\exp\Big({\int_t^\tau P(t,x)(\eta)d\eta}\Big)\leq \exp\Big(C(s^2/r)(t^{-1} - \tau^{-1})\Big)
	$$
	Thus 
	\begin{equation}\label{eq15-19-08-2022-M}
	\lim_{T\rightarrow \infty}|U_{t,x}(T)|e^{\int_t^TP(t,x)(\eta)d\eta} \leq Ce^{s^2/tr} \lim_{T\rightarrow \infty}|U_{t,x}(T)| = Ce^{s^2/tr}\mathcal{R}_{u_t}(s^2/2r,x/r).
	\end{equation}
	For the right-hand side of \eqref{eq16-19-08-2022-M}, we remark that \eqref{eq2-20-08-2022-M} and \eqref{eq14-19-08-2022-M} lead to 
	$$
	\aligned
	\int_t^T \big(S^w[u] + \Delta^w[u]\big)\big|_{\gamma(\tau;t,x)}e^{\int_t^{\tau}P_{t,x}(\eta)d\eta}d\tau
	\leq& CC_1\vep \int_t^{+\infty} \tau^{-2+\delta} e^{\int_t^{+\infty} P_{t,x}(\eta)d\eta}d\tau 
	\\
	\leq& CC_1e^{s^2/tr}\vep t^{-1+\delta}. 
	\endaligned
	$$
	Then by taking the limit $T\rightarrow +\infty$ for both sides of \eqref{eq16-19-08-2022-M}, we obtain
	\begin{equation}
	t|\del_tu(t,x)|\leq CC_1\vep e^{s^2/tr} t^{-1+\delta} + Ce^{s^2/tr}\mathcal{R}_{u_t}(s^2/2r,x/r).
	\end{equation}
	Then if $\mathcal{R}_{u_t}(s^2/2r,x/r) = 0$, we obtain, for $r/t$ sufficiently close to one
	\begin{equation}
	|\del_tu(t,x)|\leq CC_1 e^{s^2/rt}\vep t^{-2+\delta}.
	\end{equation}
	When $(r/t)>1-\eta$ with $\eta<1$, one has $e^{s^2/rt}$ is uniformly bounded. Then \eqref{eq18-19-08-2022-M} is established.
\end{proof}	
\subsection{The rigidity result}
We are now ready to establish the rigidity result.
\begin{proof}[Proof of Theorem~\ref{thm1-28-dec-2023}]
	Recalling that by the standard energy identity,
	$$
	\Ebf_0(s_1,u) = \Ebf_0(2,u)  + \int_2^{s_1}\int_{\Hcal_s} (s/t)\del_t u \Box u dxds
	$$
Then by \eqref{eq3b-13-05-2021} which holds for all $s\geq 2$,
	\begin{equation}\label{eq9-19-08-2022-M}
	\Ebf_0(2,u) - \Ebf_0(s,u)  \leq  CC_1\eps \int_2^{s_1} s^{-3/2+\delta}\Ebf_0(s,u)ds.
	\end{equation}
On the other hand, also by the standard energy estimate, 
$$
\Ebf_0(s_1,u)^{1/2}\leq \Ebf_0(2,u)^{1/2} + C\int_{2}^{s_1}(s/t)|\del u\del v|^2\,dxds \leq \Ebf_0(2,u)^{1/2} + CC_1\eps \int_{2}^{s_1}s^{-3/2+\delta}\Ebf_0(s,u)^{1/2}\,ds.
$$
By Gronwall's inequality, 
\begin{equation}\label{eq3-29-dec-2023}
\Ebf_0(s,u)^{1/2}\leq C\Ebf_0(2,u)^{1/2},\quad  \forall s\geq 2
\end{equation}
where $C$ is a universal constant, provided that $\delta<1/4$. Then we consider \eqref{eq3-29-dec-2023} and \eqref{eq9-19-08-2022-M} together, and obtain, 
$$
\Ebf_0(2,u)\leq C\Ebf_0(s,u)
$$
provided that $C_1\eps$ is sufficiently small. Then we apply Lemma~\ref{lem1-29-dec-2023} together with Lemma~\ref{lem2-29-dec-2023}, and obtain
\begin{equation}
\Ebf_0(2,u) = 0,
\end{equation}
which also leads to, regarding \eqref{eq3-29-dec-2023}, $\Ebf_0(s,u)\equiv 0$ for $s\geq 2$. Thus $u\equiv 0$ in $\Kcal_{[2,+\infty)}$.
	
To show that $u_0, u_1$ vanish on $\{t=2\}$, we will prove that $\Ebf_0^{\Pcal}(2,u) = 0$. To see this ,we firstly remark that 
$$
\Ebf_0^{\Pcal}(5/2,u) = 0.
$$
That is because $\{t=5/2\}\cap\Kcal\subset \Kcal_{[2,+\infty)}$, in which $u$ vanishes identically. On the other hand, classical energy estimate on time-constant slices leads to, thanks to \eqref{eq3b-13-05-2021},
$$
\Ebf_0^{\Pcal}(2,u) - \Ebf_0^{\Pcal}(t,u) \leq CC_1\eps\int_{2}^t\int_{\RR^3} \tau^{-3/2+\delta}\Ebf_0^{\Pcal}(\tau,u)
$$
and 
$$
\Ebf_0^{\Pcal}(t,u)^{1/2}\leq \Ebf_0^{\Pcal}(2,u)^{1/2} + C\int_2^t\tau^{-3/2+\delta}\Ebf_0^{\Pcal}(\tau,u)^{1/2}.
$$
Then by exactly the same calculation made for the hyperboloidal energy, we obtain
\begin{equation}
\Ebf_{\Pcal}(2,u)\simeq \Ebf_0^{\Pcal}(t,u).
\end{equation}
Thus , taking $t=5/2$, we obtain the desired result.
\end{proof}
	\appendix
	\section{Proof of Lemma \ref{lem1-05-11-2021}}
	We write the ODE into system
	$$
	\left(
	\begin{array}{c}
		v'
		\\
		v
	\end{array}\right)'
	+ 
	\left(
	\begin{array}{cc}
		0 &c^2(1+q(s))
		\\
		-1 &0
	\end{array}
	\right)
	\left(
	\begin{array}{c}
		v'
		\\
		v
	\end{array}\right)
	= 
	\left(
	\begin{array}{c}
		f
		\\
		0
	\end{array}
	\right).
	$$
	Let $Q = \left(
	\begin{array}{cc}
		ci\sqrt{1+q} &0
		\\
		0 &-ci\sqrt{1+q}
	\end{array}
	\right)$, the above system is diagonalized as 
	\begin{equation}\label{eq1-11-05-2021}
		V'(s) + P(s)Q(s)P^{-1}(s)V(s) = F(s),\quad V(s) = (v'(s),v(s))^T,\quad F(s) = (f(s),0)^T 
	\end{equation}
	and
	$$
	P=P(s) = 
	\left(
	\begin{array}{cc}
		-ci\sqrt{1+q(s)} &ci\sqrt{1+q(s)}
		\\
		1 &1
	\end{array}
	\right),
	\quad
	P^{-1}(s) = 
	\left(
	\begin{array}{cc}
		\frac{-1}{2ci\sqrt{1+q(s)}} &\frac{1}{2}
		\\
		\frac{1}{2ci\sqrt{1+q(s)}} &\frac{1}{2}
	\end{array}
	\right).
	$$
	When $|q|\leq 1/2$, $|P(s)|\lesssim (1+c)$ and $|P^{-1}(s)|\lesssim 1+c^{-1}$  on $[s_0,s_1]$. Now write \eqref{eq1-11-05-2021} into the following form
	$$
	(P^{-1}V)' + Q(P^{-1}V) = P^{-1} F + (P^{-1})'V 
	$$
	where we remark that
	$$
	(P^{-1})'(s)V(s) = \left(
	\begin{array}{cc}
		(4ci)^{-1}(1+q(s))^{-3/2}q'(s)v'(s)
		\\
		-(4ci)^{-1}(1+q(s))^{-3/2}q'(s)v'(s) 
	\end{array}
	\right)
	,\quad
	P^{-1}F =
	\left( \begin{array}{c}
		\frac{-f(s)}{2ci\sqrt{1+q(s)}}
		\\
		\frac{f(s)}{2ci\sqrt{1+q(s)}}
	\end{array}
	\right).
	$$
	Recall that $Q$ is diagonalized with pure imaginary elements. Then
	\begin{equation}
		|P^{-1}(V(s)-V(s_0))| \leq  Cc^{-1}\int_{s_0}^s|f(s^{\prime})| + |q'(s)v'(s)|\, ds^{\prime}.
	\end{equation}
	where $C$ is a universal constant. This guarantees \eqref{eq2-11-05-2021}.

\end{document}